\documentclass[10pt]{article}
\usepackage{amsmath, amssymb, amsthm}
\usepackage{bbm}  
\usepackage{accents}  
\usepackage{cleveref}
\usepackage{multicol}
\usepackage{tikz, float} 
\usepackage{geometry}
\title{Runs in Random Sequences over Ordered Sets}
\author{Tanner Reese \\ \small{University of Arizona}}

\newtheorem{thm}{Theorem}
\newtheorem{lem}[thm]{Lemma}
\newtheorem{prop}[thm]{Proposition}
\newtheorem{coro}[thm]{Corollary}

\theoremstyle{definition}
\newtheorem{defn}{Definition}

\newcommand\abs[1]{\left|#1\right|}
\newcommand\norm[1]{\left\|#1\right\|}

\newcommand\prob[2][]{\mathbb{P}_{#1}\left(#2\right)}
\newcommand\probWhen[3][]{\mathbb{P}_{#1}\left(#2 \,\middle|\, #3\right)}
\newcommand\expect[2][]{\mathbb{E}_{#1}\left[#2\right]}
\newcommand\expectWhen[3][]{\mathbb{E}_{#1}\left[#2 \,\middle|\, #3\right]}
\newcommand\var[2][]{\mathrm{var}_{#1}\left(#2\right)}
\newcommand\varWhen[3][]{\mathrm{var}_{#1}\left(#2 \,\middle|\, #3\right)}

\newcommand\borel[1]{\mathcal{B}_{#1}}
\newcommand\dr{\; \mathrm{d}}

\newcommand\closed[1]{\accentset{\bullet}{#1}}
\newcommand\open[1]{\accentset{\circ}{#1}}

\begin{document}
\maketitle

\begin{abstract}
We determine the distributions of lengths of runs in random sequences of elements from a
totally ordered set (total order) or partially ordered set (partial order).
In particular, we produce novel formulae for the expected value, variance, and probability generating function (PGF)
of such lengths in the case of an arbitrary total order.
Our focus is on the case of distributions with both atoms and diffuse (absolutely or singularly continuous) mass
which has not been addressed in this generality before.
We also provide a method of calculating the PGF of run lengths for countably series-parallel partial orders.
Additionally, we prove a strong law of large numbers for the distribution of run lengths in a particular realization of an infinite sequence.
\end{abstract}
 
\section{Introduction}
\label{sect:intro}

A \textbf{run} in a sequence $ \{a_i\}_{i=0}^\infty $ is a maximal subsequence of consecutive monotonic elements.
We say $ \langle a_i, \ldots, a_{i+n-1} \rangle $ is a \textbf{strict run} when
\[ \begin{split}
a_{i-1} &\nless a_i < a_{i+1} < \ldots < a_{i+n-1} \nless a_{i+n} \quad \text{ or } \\
a_{i-1} &\ngtr a_i > a_{i+1} > \ldots > a_{i+n-1} \ngtr a_{i+n}
\end{split} \]
and $ \langle a_i, \ldots, a_{i+n-1} \rangle $ is a \textbf{non-strict run} when
\[ \begin{split}
a_{i-1} &\nleq a_i \leq a_{i+1} \leq \ldots \leq a_{i+n-1} \nleq a_{i+n} \quad \text{ or } \\
a_{i-1} &\ngeq a_i \geq a_{i+1} \geq \ldots \geq a_{i+n-1} \ngeq a_{i+n}
\end{split} \]
where $ a_{i-1} $ or $ a_{i+n} $ may not exist if the run starts at zero or extends to infinity.
Here, we say the runs begin at index $ i $ and have length $ n $ (if the run extends to infinity we say $ n = \infty $).
When the values are ascending (resp.\ non-descending), the run is a strict (resp.\ non-strict) \textbf{run up}
and if the values are descending (resp.\ non-ascending) then it is a strict (resp.\ non-strict) \textbf{run down}.
Because runs up and down may be swapped by reversing the ordering relation,
we will just consider runs up for the remainder of the paper.
\medskip

Suppose $ T $ is a partial order and $ \mu $ is a probability measure on $ T $.
Let $ \{X_i\}_{i=0}^\infty $ be a sequence of i.i.d.\ random variables
on the probability space $ \langle \Omega, \mathcal{F}, \mathbb{P} \rangle $ distributed according to $ \mu $.
We are interested in the distribution of the lengths of runs in the sequence $ X_0, X_1, \ldots $
particularly when $ \mu $ has atoms and/or $ T $ is a countably series-parallel partial order
(definition in \cref{sub:series-parallel}).
\medskip

Runs have been studied extensively in the context of permutations beginning with Andr\'e \cite{DAnd}.
Of particular interest has been the enumeration of permutations with a given number of runs
and the generating functions associated with these numbers \cite{Can, Ma, Chen}.
A general treatment of runs and other patterns in permutations can be found in B\'ona's book \cite{Bon}
with a review of probabilistic methods in chapter six.
The distribution of the length of runs in permutations can be used to calculate
the distribution of the length of runs in a sequence when $ \mu $ is diffuse (continuous).
However, since permutations do not allow duplicate entries,
the behavior of runs in permutations will not reflect the behavior of
runs in sequences where $ \mu $ contains atoms (and perhaps diffuse mass).
\medskip

Runs in random sequences are also closely related to ascents, descents,
and records (running maxima and minima) which have been explored extensively.
The asymptotic theory of order statistics has been well studied.
For example, a theorem of R\'enyi \cite{Ren} regarding records
can be used to calculate the distribution of run lengths
when $ \mu $ is diffuse (continuous).
However, there are fewer results in the non-diffuse case.
Existing theory for records in the non-diffuse case
is focused on the mean, variance, and position of records for discrete distributions,
but does not address runs or mixed distributions
(see \cite{Key}, chapter 6 of \cite{AN}, or for a survey \cite{Nag}).
In this paper, we will be considering not just the discrete case,
but cases in which the distribution contains both discrete and continuous parts.
We will also consider a broad class of partial orders which we call countably series-parallel
(see \cref{sub:series-parallel} for a definition).
\medskip

Research into runs in permutations has found practical applications in computer science.
The number of runs of a given length in a permutation was used as a test of randomness in section 3.3.2 of \cite{Knu2}.
Runs and their lengths are also important for sorting algorithms as described in section 5.1.3 of \cite{Knu3}.
Given these applications, it seems valuable to investigate runs in the context of i.i.d.\ random sequences
where the distribution may contain atoms.
\medskip

We will now define the run length as a random variable.
For each index $ i \geq 0 $, we define $ \open{S}_i $ (resp.\ $ \closed{S}_i $)
as the event that a strict (resp.\ non-strict) run begins at $ i $.
More precisely, the events will be
\[ \begin{split}
\open{S}_i := \begin{cases}
\Omega & \text{ when } i = 0 \\
\{X_{i-1} \nless X_i\} & \text{ when } i \geq 1
\end{cases}
\quad \text{ and } \quad
\closed{S}_i := \begin{cases}
\Omega & \text{ when } i = 0 \\
\{X_{i-1} \nleq X_i\} & \text{ when } i \geq 1
\end{cases}.
\end{split} \]
We also define the random variable $ \open{N}_i $ (resp.\ $ \closed{N}_i $)
as the length of the strict (resp.\ non-strict) run beginning at $ i $ assuming one is present.
More precisely, for any $ n \in \{1, 2, \ldots, \infty\} $,
\[ \begin{split}
\left\{ \open{N}_i = n \right\} &:= \begin{cases}
\{X_0 < \ldots < X_{n-1} \nless X_{i+n}\} & \text{ when } i = 0 \text{ and } n < \infty \\
\{X_0 < X_1 < \ldots\} & \text{ when } i = 0 \text{ and } n = \infty \\
\{X_{i-1} \nless X_i < \ldots < X_{i+n-1} \nless X_{i+n} \} & \text{ when } i \geq 1 \text{ and } n < \infty \\
\{X_{i-1} \nless X_i < X_{i+1} < \ldots\} & \text{ when } i \geq 1 \text{ and } n = \infty
\end{cases} \\
\left\{ \closed{N}_i = n \right\} &:= \begin{cases}
\{X_0 \leq \ldots \leq X_{n-1} \nleq X_{i+n}\} & \text{ when } i = 0 \text{ and } n < \infty \\
\{X_0 \leq X_1 \leq \ldots\} & \text{ when } i = 0 \text{ and } n = \infty \\
\{X_{i-1} \nleq X_i \leq \ldots \leq X_{i+n-1} \nleq X_{i+n} \} & \text{ when } i \geq 1 \text{ and } n < \infty \\
\{X_{i-1} \nleq X_i \leq X_{i+1} \leq \ldots\} & \text{ when } i \geq 1 \text{ and } n = \infty
\end{cases}.
\end{split} \]
Notice that $ \open{N}_i $ and $ \closed{N}_i $ are \emph{not} defined on the entire probability space $ \Omega $.
Instead, they are restricted to $ \open{S}_i $ and $ \closed{S}_i $, respectively.
In order for $ \open{N}_i $ and $ \closed{N}_i $ to be well-defined, $ \{X_0 < X_1\} $ must be measurable.
We will assume that this is the case throughout the rest of the paper.
The probability generating functions (PGFs) of $ \open{N}_i $ and $ \closed{N}_i $ are the formal power series
\[
G_{\open{N}_i}(Z) := \sum_{n=0}^\infty \probWhen{\open{N}_i = n}{\open{S}_i} Z^n
\quad \text{ and } \quad
G_{\closed{N}_i}(Z) := \sum_{n=0}^\infty \probWhen{\closed{N}_i = n}{\closed{S}_i} Z^n.
\]
Our main interest in this paper will be calculating the PGFs, $ G_{\open{N}_i} $ and $ G_{\closed{N}_i} $,
and from these determining the means and variances of the run lengths.
Because of the generality of $ T $ and $ \mu $, there are many examples of random sequences
that can be modelled this way.
\begin{itemize}
\item Suppose we have a block of radioactive material
and we let $ X_i $ be the amount of time between the $ i $-th and $ (i+1) $-th decay.
We can take $ T = \mathbb{R}_{\geq 0} $ and $ \mu $ is exponential with $ \mu([0, t)) = 1 - e^{-\lambda t} $.
Then, $ \open{N}_i $ (resp.\ $ \closed{N}_i $) is the number of decay events starting with the $ i $-th
for which the amount of time between decays is strictly increasing (resp.\ non-decreasing).

\item Suppose we have an evenly weighted $ n $-sided die.
We let $ X_i $ represent the value of the $ i $-th roll.
In this case, $ T = \{1, \ldots, n\} $ and $ \mu(\{k\}) = \frac{1}{n} $ for each $ k \in T $.
Then, $ \open{N}_i $ (resp.\ $ \closed{N}_i $) is the number of rolls starting with the $ i $-th
for which the value is strictly increasing (resp.\ non-decreasing).

\item Suppose we have a dart board with a bullseye.
For any two throws that land outside the bullseye, we say the one that is closer to the bullseye is better.
However, for any throws landing in the bullseye, we say they are all equally good.
The total order $ T $ can be described as $ (0, 1] $.
Here, $ 1 $ represents a throw landing in the bullseye and $ (0, 1) $ represents throws landing outside
with $ x > y $ when $ x $ is closer to the bullseye.
The probability measure $ \mu $ assigns mass $ p \in (0, 1) $ to the bullseye $ \mu(\{1\}) = p $
and the remaining mass is uniformly distributed over $ (0, 1) $ so $ \mu(A) = (1 - p) m(A) $
for any Borel set $ A $ where $ m $ is the Lebesgue measure.
Then, $ \open{N}_i $ (resp.\ $ \closed{N}_i $) is the number of throws starting with the $ i $-th throw
for which each throw is better (resp.\ not worse) than the previous.

\item Suppose there is a light bulb socket
whose bulb has a probability $ p \in (0, 1) $ of going out each day.
We keep track of the number of days each light bulb lasts.
If we let $ X_i $ be the number of days the $ i $-th bulb lasts
then $ T = \{1, 2, \ldots\} $ and $ \mu $ is geometric with $ \mu(\{k\}) = p (1 - p)^{k-1} $.
Then, $ \open{N}_i $ (resp.\ $ \closed{N}_i $) is the number of bulbs starting with the $ i $-th bulb
for which every bulb lasts longer than (resp.\ at least as long as) the previous.

\item Suppose one has an evenly weighted $ 2n $-sided die.
We want to require that runs are not just increasing (non-decreasing),
but that the elements of a run are all of the same parity.
That is the sequence $ 1, 3, 7, 8, 2, 4, 5, 9 $ would have runs $ 1, 3, 7 \;|\; 8 \;|\; 2, 4 \;|\; 5, 9 $.
This can be achieved by taking $ T = \{1, 2, \ldots, 2n\} $ to be a partial order where
\[ 1 < 3 < \ldots < 2n - 1 \quad \text{ and } \quad 2 < 4 < \ldots < 2n \]
and $ i, j \in T $ are incomparable whenever $ i $ and $ j $ are of different parity.
As above, we take $ \mu(\{k\}) = \frac{1}{2n} $ for each $ k \in T $.
Then, $ \open{N}_i $ (resp.\ $ \closed{N}_i $) represents the number of rolls starting with the $ i $-th roll
for which the value is strictly increasing (resp.\ non-decreasing) and the parity is the same.
\end{itemize}
We will present a method of calculating the PGFs
for $ \open{N}_i $ and $ \closed{N}_i $ in all of these cases.
The PGFs, means, and variances for these examples can be found in \cref{sect:examples}.
This method is applicable whenever $ T $ is countably series-parallel (which includes total orders)
and $ \mu $ is a probability measure (potentially containing atoms).
\medskip

I would like to thank Sunder Sethuraman for helping me with the editing of this paper.
The writing of this paper was partially supported by a Galileo Circle Scholarship from the College of Science at the University of Arizona.
\medskip

\section{Results}

This section contains the statements of the theorems and corollaries.
Their proofs can be found in the later sections.
First, we will state our primary results for total orders
and then, in \cref{sub:series-parallel}, we will consider partial corders.

\subsection{Total Orders}
\label{sub:total-ords}

We say $ x \in T $ is an \textbf{atom} of $ \mu $ if $ \mu(\{x\}) > 0 $.
Let $ \mathcal{A} $ be the set of atoms of $ \mu $.
For each $ \alpha \in \mathcal{A} $, we say $ m_\alpha = \mu(\{\alpha\}) $.
The mass of $ \mu $ not containing atoms is the \textbf{diffuse mass}, $ m_d = \mu(T \setminus \mathcal{A}) $.
We will be considering $ \mu $ that contain both atoms and diffuse mass.
In particular, when $ T = \mathbb{R} $, this means that $ \mu $
may contain discrete, singularly continuous, and absolutely continuous sections.
If there exists $ x \in T $ such that $ \mu(\{x\}) = \mu(T) $ then we say $ \mu $ is \textbf{degenerate}
and otherwise $ \mu $ is \textbf{non-degenerate}.

\begin{thm} \label{thm:total-length-pgf}
If $ T $ is a total order and $ \mu $ is a non-degenerate probability measure on $ T $ then
for all $ z \in \mathbb{C} \setminus \{0\} $,
the PGFs of $ \open{N}_i $ and $ \closed{N}_i $ will fulfill
\begin{equation} \label{eq:total-init-length-pgf}
G_{\open{N}_0}(z) = \frac{1}{z} \left(1 + (z - 1) e^{m_d z}
\prod_{\alpha \in \mathcal{A}} (1 + m_\alpha z)
\right)
\quad \text{ and for } |z| \leq 1, \quad
G_{\closed{N}_0}(z) = \frac{1}{z} \left(1 + (z - 1) e^{m_d z}
\prod_{\alpha \in \mathcal{A}} \frac{1}{1 - m_\alpha z}
\right)
\end{equation}
\begin{equation} \label{eq:total-length-pgf}
G_{\open{N}_i}(z) = 1 + \frac{2 (z - 1)}{z \left(1 + \sum_{\alpha \in \mathcal{A}} m_\alpha^2 \right)} G_{\open{N}_0}(z)
\quad \text{ and for } |z| \leq 1, \quad
G_{\closed{N}_i}(z) = 1 + \frac{2 (z - 1)}{z \left(1 - \sum_{\alpha \in \mathcal{A}} m_\alpha^2 \right)} G_{\closed{N}_0}(z)
\text{ when } i \geq 1.
\end{equation}
\end{thm}

From the PGFs, it is possible to immediately calculate the means and variances.

\begin{coro} \label{coro:total-length-stats}
For any total order $ T $ and non-degenerate probability measure $ \mu $ on $ T $,
the means of the run lengths will be
\[
\expect{\open{N}_0} = -1 + e^{m_d} \prod_{\alpha \in \mathcal{A}} (1 + m_\alpha)
\quad \text{ and } \quad
\expect{\closed{N}_0} = -1 + e^{m_d} \prod_{\alpha \in \mathcal{A}} \frac{1}{1 - m_\alpha}
\]
\[
\expectWhen{\open{N}_i}{\open{S}_i} = \frac{2}{1 + \sum_{\alpha \in \mathcal{A}} m_\alpha^2}
\quad \text{ and } \quad
\expectWhen{\closed{N}_i}{\closed{S}_i} = \frac{2}{1 - \sum_{\alpha \in \mathcal{A}} m_\alpha^2}
\quad \text{ when } i \geq 1.
\]
Further, the variances of the run lengths will be
\begin{equation} \label{eq:total-init-length-var} \begin{split}
\var{\open{N}_0} &= \left( e^{m_d} \prod_{\alpha \in \mathcal{A}} (1 + m_\alpha) \right)
\left(3 - e^{m_d} \prod_{\alpha \in \mathcal{A}} (1 + m_\alpha) - 2 \sum_{\alpha \in \mathcal{A}} \frac{m_\alpha^2}{1 + m_\alpha}\right) \\
\var{\closed{N}_0} &= \left( e^{m_d} \prod_{\alpha \in \mathcal{A}} \frac{1}{1 - m_\alpha} \right)
\left(3 - e^{m_d} \prod_{\alpha \in \mathcal{A}} \frac{1}{1 - m_\alpha} + 2 \sum_{\alpha \in \mathcal{A}} \frac{m_\alpha^2}{1 - m_\alpha} \right)
\end{split} \end{equation}
\begin{equation} \label{eq:total-length-var} \begin{split}
\varWhen{\open{N}_i}{\open{S}_i} &= \frac{2}{1 + \sum_{\alpha \in \mathcal{A}} m_\alpha^2}
\left(-3 + 2 e^{m_d} \prod_{\alpha \in \mathcal{A}} (1 + m_\alpha) - \frac{2}{1 + \sum_{\alpha \in \mathcal{A}} m_\alpha^2}\right)
\quad \text{ for } i \geq 1 \\
\varWhen{\closed{N}_i}{\closed{S}_i} &= \frac{2}{1 - \sum_{\alpha \in \mathcal{A}} m_\alpha^2}
\left(-3 + 2 e^{m_d} \prod_{\alpha \in \mathcal{A}} \frac{1}{1 - m_\alpha} - \frac{2}{1 - \sum_{\alpha \in \mathcal{A}} m_\alpha^2} \right)
\quad \text{ for } i \geq 1.
\end{split} \end{equation}
\end{coro}

\subsection{Countably Series-Parallel Partial Orders}
\label{sub:series-parallel}

In this subsection, we will introduce a method for calculating the PGFs of the run lengths
whenever $ T $ is a countably series-parallel partial order (defined below).
First, however, we will state some results which apply to all partial orders.
To do this, we will introduce two very useful generating functions using $ Z $ as a formal indeterminate.
For any finite non-negative measure $ \mu $ on a partial order $ T $,
we define the \textbf{strict run function} $ \open{P}_\mu(Z) $ as
\begin{equation} \label{eq:strict-run-func-defn}
\open{P}_\mu(Z) := 1 + \mu(T) Z + \sum_{n=2}^\infty \mu^n(\{\langle a_1, \ldots, a_n \rangle \in T^n \;:\; a_1 < \ldots < a_n\}) Z^n
\end{equation}
and the \textbf{non-strict run function} $ \closed{P}_\mu(Z) $ as
\begin{equation} \label{eq:non-strict-run-func-defn}
\closed{P}_\mu(Z) := 1 + \mu(T) Z + \sum_{n=2}^\infty \mu^n(\{\langle a_1, \ldots, a_n \rangle \in T^n \;:\; a_1 \leq \ldots \leq a_n\}) Z^n.
\end{equation}
In particular, when $ \mu $ is a probability measure and $ \{X_i\} $ is defined as above, we have
\[
\open{P}_\mu(Z) = 1 + Z + \sum_{n=2}^\infty \prob{X_1 < \ldots < X_n} Z^n
\quad \text{ and } \quad
\closed{P}_\mu(Z) = 1 + Z + \sum_{n=2}^\infty \prob{X_1 \leq \ldots \leq X_n} Z^n.
\]
The usefulness of these generating functions is justified by the following result
which holds for \emph{any} partial order.

\begin{thm} \label{thm:length-pgf}
For any partial order $ T $ and probability measure $ \mu $ on $ T $,
the PGFs of $ \open{N}_i $ and $ \closed{N}_i $ will be
\begin{equation} \label{eq:init-length-pgf}
G_{\open{N}_0}(Z) := \frac{1 + (Z - 1) \open{P}_\mu(Z)}{Z}
\quad \text{ and } \quad
G_{\closed{N}_0}(Z) := \frac{1 + (Z - 1) \closed{P}_\mu(Z)}{Z}.
\end{equation}
\begin{equation} \label{eq:length-pgf}
G_{\open{N}_i}(Z) = 1 + \frac{2 (Z - 1)}{Z (2 - \open{P}_\mu''(0))} G_{\open{N}_0}(Z)
\quad \text{ and } \quad
G_{\closed{N}_i}(Z) = 1 + \frac{2 (Z - 1)}{Z (2 - \closed{P}_\mu''(0))} G_{\closed{N}_0}(Z)
\text{ when } i \geq 1.
\end{equation}
Further, the radius of convergence of $ G_{\open{N}_i}(Z) $ will be infinite
and the radius of convergence of $ G_{\closed{N}_i}(Z) $ will be greater than or equal to $ \frac{1}{\|\mu\|} $
with equality when $ \mu $ is degenerate.
\end{thm}

As before, we can use the PGFs to derive formulae for the means and variances of the run lengths.

\begin{coro} \label{coro:length-stats}
Suppose $ T $ is a partial order and $ \mu $ is a non-degenerate probability measure over $ T $.
The run lengths $ \open{N}_i $ and $ \closed{N}_i $ are almost surely finite and their means are
\begin{equation} \label{eq:init-length-mean}
\expect{\open{N}_0} = \open{P}_\mu(1) - 1
\quad \text{ and } \quad
\expect{\closed{N}_0} = \closed{P}_\mu(1) - 1
\end{equation}
\begin{equation} \label{eq:length-mean}
\expectWhen{\open{N}_i}{\open{S}_i} = \frac{2}{2 - \open{P}''(0)}
\quad \text{ and } \quad
\expectWhen{\closed{N}_i}{\closed{S}_i} = \frac{2}{2 - \closed{P}''(0)}
\quad \text{ for } i \geq 1.
\end{equation}
Also, the variances of the run lengths will be
\begin{equation} \label{eq:init-length-var}
\var{\open{N}_0} = \open{P}_\mu(1) - \open{P}_\mu(1)^2 + 2 \open{P}_\mu'(1)
\quad \text{ and } \quad
\var{\closed{N}_0} = \closed{P}_\mu(1) - \closed{P}_\mu(1)^2 + 2 \closed{P}_\mu'(1)
\end{equation}
\begin{equation} \label{eq:length-var}
\varWhen{\open{N}_i}{\open{S}_i} = \frac{4 \open{P}_\mu(1) - 6}{2 - \open{P}_\mu''(0)} - \frac{4}{(2 - \open{P}_\mu''(0))^2}
\quad \text{ and } \quad
\varWhen{\closed{N}_i}{\closed{S}_i} = \frac{4 \closed{P}_\mu(1) - 6}{2 - \closed{P}_\mu''(0)} - \frac{4}{(2 - \closed{P}_\mu''(0))^2}
\quad \text{ for } i \geq 1.
\end{equation}
\end{coro}

While \cref{thm:length-pgf} holds for any partial order,
it remains to actually calculate the run functions.
In the case of total orders, we have the following formulae for the run functions.

\begin{prop} \label{prop:total-run-funcs}
Suppose $ T $ is a total order and $ \mu $ is a finite non-negative and non-degenerate measure on $ T $.
Let $ \mathcal{A} $ be the atoms of $ \mu $ with masses $ (m_\alpha)_{\alpha \in \mathcal{A}} $
and $ m_d $ be the diffuse mass of $ \mu $.
Then, for any $ z \in \mathbb{C} $,
\begin{equation} \label{eq:total-run-funcs}
\open{P}_\mu(z) = e^{m_d z} \prod_{\alpha \in \mathcal{A}} (1 + m_\alpha z)
\quad \text{ and when } |z| \leq \frac{1}{\|\mu\|}, \quad
\closed{P}_\mu(z) = e^{m_d z} \prod_{\alpha \in \mathcal{A}} \frac{1}{1 - m_\alpha z}.
\end{equation}
\end{prop}

Using \cref{prop:total-run-funcs} together with \cref{thm:length-pgf} and \cref{coro:length-stats},
we can derive the results for total orders described in \cref{thm:total-length-pgf} and \cref{coro:total-length-stats}.
We can also calculate the run functions in the case of countably series-parallel partial orders.
In order to define this class of partial orders, we will introduce the following operations.

\begin{defn} \label{defn:series-para-ord}
Suppose $ \Gamma $ is a totally ordered index set
and $ \mathcal{T} = \{T_\gamma\}_{\gamma \in \Gamma} $ is an indexed collection of partial orders.
We define the \textbf{series composition} and \textbf{parallel composition} of $ \mathcal{T} $
as the partial orders
\[
\bigotimes \mathcal{T} = \bigotimes_{\gamma \in \Gamma} T_\gamma \quad \text{ and } \quad
\bigoplus \mathcal{T} = \bigoplus_{\gamma \in \Gamma} T_\gamma,
\]
respectively.
The underlying set of $ \bigotimes \mathcal{T} $ and $ \bigoplus \mathcal{T} $ is the disjoint union $ \bigsqcup_{\gamma \in \Gamma} T_\gamma $.
For any $ \gamma, \tilde{\gamma} \in \Gamma $ with $ \gamma \neq \tilde{\gamma} $, the ordering relation of $ \bigotimes \mathcal{T} $ is
\[ \begin{split}
\text{for all }& t_1, t_2 \in T_\gamma, \qquad t_1 \leq_{\bigotimes \mathcal{T}} \, t_2 \iff t_1 \leq_{T_\gamma} t_2 \quad \text{ and} \\
\text{for all }& t \in T_\gamma, \tilde{t} \in T_{\tilde{\gamma}}, \qquad
t \leq_{\bigotimes \mathcal{T}} \, \tilde{t} \iff \gamma <_\Gamma \tilde{\gamma}.
\end{split} \]
and the ordering relation of $ \bigoplus \mathcal{T} $ is
\[ \begin{split}
\text{for all }& t_1, t_2 \in T_\gamma, \qquad t_1 \leq_{\bigoplus \mathcal{T}} \, t_2 \iff t_1 \leq_{T_\gamma} t_2 \quad \text{ and} \\
\text{for all }& t \in T_\gamma, \tilde{t} \in T_{\tilde{\gamma}}, \qquad t \text{ is incomparable to } \tilde{t}.
\end{split} \]
Note that the ordering on $ \Gamma $ has no effect on the parallel composition.
\end{defn}

If $ T_1 $ and $ T_2 $ are partial orders, we can think of $ T_1 \otimes T_2 $ as the partial order
in which we place all of the elements of $ T_1 $ below all of the elements of $ T_2 $.
Simiarly, $ T_1 \oplus T_2 $ is the partial order
in which we place all of the elements of $ T_1 $ beside all of the elements of $ T_2 $
so that every $ x \in T_1 $ is incomparable to every $ y \in T_2 $.
The above operations have been studied and rediscovered a number of times.
As a result, there are a variety of different notations and names for them.
The above ``series composition" is sometimes referred to as the
``linear sum", ``ordinal sum", or ``lexicographic sum".
Furthermore the above ``parallel composition" is sometimes referred to as the
``direct sum" or ``disjoint union".
In most cases, only series and parallel compositions where $ \Gamma $ is finite are considered.
Using the terms series and parallel composition,
M\"ohring enumerates several of the main results about them particularly in the case of finite orders \cite{Moh}.
Schr\"oder provides a more abstract treatment of series composition using the term lexicographic sum \cite{Sch}.
\medskip

A natural collection of partial orders to consider is those which can be constructed using these operations.
We will say that a partial order $ T $
is \textbf{countably series-parallel} if it can be constructed
from total orders using countable series and parallel compositions
(ie.\ where the index set $ \Gamma $ is countable).
Note that the series-parallel partial orders considered by M\"ohring \cite{Moh} and other authors
only allow finite series and parallel compositions starting with singleton orders.
Next, we will extend series and parallel composition to measures over orders
which does not appear to have been done before.

\begin{defn} \label{defn:series-para-meas}
Let $ \Gamma $ and $ \mathcal{T} $ be as in \cref{defn:series-para-ord}.
Suppose $ M = \{\mu_\gamma\}_{\gamma \in \Gamma} $ is a collection of finite non-negative measures
where $ \mu_\gamma $ is defined over the Borel $ \sigma $-algebra of $ T_\gamma $.
Further, we assume that $ \sum_{\gamma \in \Gamma} \mu_\gamma(T_\gamma) < \infty $.
Then, we define the \textbf{series composition} and \textbf{parallel composition} of $ M $
as the finite non-negative measures
\[
\bigotimes M = \bigotimes_{\gamma \in \Gamma} \mu_\gamma \quad \text{ and } \quad
\bigoplus M = \bigoplus_{\gamma \in \Gamma} \mu_\gamma,
\]
respectively.
Then, for any $ E \in \borel{\bigotimes \mathcal{T}} $ and $ F \in \borel{\bigoplus \mathcal{T}} $,
\[
\left(\bigotimes M\right)(E) = \sum_{\gamma \in \Gamma} \mu_\gamma(E \cap T_\gamma) \quad \text{ and } \quad
\left(\bigoplus M\right)(F) = \sum_{\gamma \in \Gamma} \mu_\gamma(F \cap T_\gamma).
\]
\end{defn}

Suppose $ \mu_1 $ and $ \mu_2 $ are measures on $ T_1 $ and $ T_2 $.
As above, we can think of $ \mu_1 \otimes \mu_2 $ as the measure
in which we place the mass of $ \mu_1 $ below the mass of $ \mu_2 $
and $ \mu_1 \oplus \mu_2 $ as the measure
in which we place the mass of $ \mu_1 $ besides the mass of $ \mu_2 $.

\begin{figure}[H] \label{fig:series-parallel} \centering
\begin{tikzpicture}
\fill (-3.5, 1.5) circle (3pt);
\fill (-3, 1.5) circle (5pt);
\draw[line width=0.5mm, line cap=round] (-2.5, 1.5) -- (-1,1.5);
\draw (-2.35, 1.5) ellipse [x radius=47pt, y radius=14pt];
\node at (-2.35, 2.25) {$ \langle \mu_1, T_1 \rangle $};

\fill (1, 1.5) circle (3pt);
\fill (1.5, 1.5) circle (3pt);
\draw[line width=0.5mm, line cap=round] (2, 1.5) -- (3, 1.5);
\fill (3.5, 1.5) circle (5pt);
\draw (2.3, 1.5) ellipse [x radius=48pt, y radius=14pt];
\node at (2.3, 2.25) {$ \langle \mu_2, T_2 \rangle $};

\draw[line width=0.4mm, ->] (0, 1) -- (-3, -0.55);
\draw[line width=0.4mm, ->] (0, 1) -- (2.6, 0);

\fill (-6.5, -2) circle (3pt);
\fill (-6, -2) circle (5pt);
\draw[line width=0.6mm, line cap=round] (-5.5, -2) -- (-4, -2);
\fill (-3.5, -2) circle (3pt);
\fill (-3, -2) circle (3pt);
\draw[line width=0.6mm, line cap=round] (-2.5, -2) -- (-1.5, -2);
\fill (-1, -2) circle (5pt);
\draw (-3.7, -2) ellipse[x radius=90pt, y radius=17pt];
\node at (-3.7, -1.1) {$ \langle \mu_1 \otimes \mu_2, T_1 \otimes T_2 \rangle $};

\fill (2, -1.5) circle (3pt);
\fill (2.5, -1.5) circle (5pt);
\draw[line width=0.6mm, line cap=round] (3, -1.5) -- (4.5, -1.5);
\fill (2, -2.5) circle (3pt);
\fill (2.5, -2.5) circle (3pt);
\draw[line width=0.6mm, line cap=round] (3, -2.5) -- (4, -2.5);
\fill (4.5, -2.5) circle (5pt);
\draw (3.3, -2) ellipse[x radius=50pt, y radius=36pt];
\node at (3.3, -0.45) {$ \langle \mu_1 \oplus \mu_2, T_1 \oplus T_2 \rangle $};
\end{tikzpicture}
\caption{
The series and parallel compositions of the measures $ \mu_1 $ and $ \mu_2 $
and their orders $ T_1 $ and $ T_2 $.
Circles represent atoms with the size of the circle indicating the mass.
Lines represent totally ordered sections of diffuse mass
with the length indicating the mass.
Lower elements are to the left.}
\end{figure}
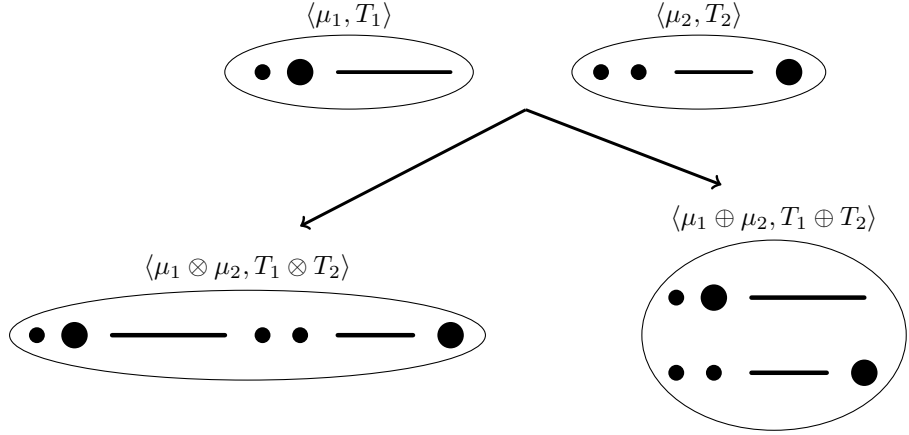

The following theorem specifies how we can compute the run functions
for a measure on a countably series-parallel partial order.
After obtaining the run functions, we can use \cref{thm:length-pgf} and \cref{coro:length-stats}
to get the PGFs, means, and variances of the run lengths.

\begin{thm} \label{thm:series-parallel}
Suppose $ T $ is a countably series-parallel partial order
and $ \mu $ is a finite non-negative measure on $ T $.
Let $ R $ be the radius of convergence of $ \closed{P}_\mu(Z) $.
Then, either $ T $ is a total order or $ T $ is a countable series or parallel composition.
If $ T $ is a total order then the run functions are determined by \cref{prop:total-run-funcs}.
If $ T $ is a series composition then $ T = \bigotimes_{\gamma \in \Gamma} T_\gamma $
and for all $ z \in \mathbb{C} $,
\[
\open{P}_\mu(z) = \prod_{\gamma \in \Gamma} \open{P}_{\mu \big|_{T_\gamma}}(z)
\quad \text{ and when } |z| < R, \quad
\closed{P}_\mu(z) = \prod_{\gamma \in \Gamma} \closed{P}_{\mu \big|_{T_\gamma}}(z)
\]
where $ \mu \big|_{T_\gamma} $ is the restriction of $ \mu $ to $ T_\gamma $.
If $ T $ is a parallel composition then $ T = \bigoplus_{\gamma \in \Gamma} T_\gamma $
and for all $ z \in \mathbb{C} $,
\[
\open{P}_\mu(z) - 1 = \sum_{\gamma \in \Gamma} \left(\open{P}_{\mu \big|_{T_\gamma}}(z) - 1 \right)
\quad \text{ and when } |z| < R, \quad
\closed{P}_\mu(z) - 1 = \sum_{\gamma \in \Gamma} \left(\closed{P}_{\mu \big|_{T_\gamma}}(z) - 1 \right).
\]
\end{thm}

\subsection{Strong Law of Large Numbers}
\label{sub:slln}

Let $ \omega \in \Omega $ be a particular outcome
and consider the proportion of strict (resp.\ non-strict) runs in the sequence $ \{X_i(\omega)\}_{i=0}^\infty $ of length $ n $.
Because $ \open{S}_i $ and $ \open{S}_j $ are correlated events, it is not clear whether this proportion will converge.
However, it happens that as $ |i - j| \to \infty $, $ \open{S}_i $ and $ \open{S}_j $
decorrelate quickly enough that this proportion will indeed converge.
In particular, the proportion of strict (resp.\ non-strict) runs
will approach $ \probWhen{\open{N}_i = n}{\open{S}_i} $ \Big(resp.\ $ \probWhen{\closed{N}_i = n}{\closed{S}_i} $ \Big) for any $ i \geq 1 $.
That is to say the asymptotic distribution of lengths of runs in a particular infinite sequence
is also given by \cref{thm:length-pgf}.

\begin{thm} \label{thm:slln}
Suppose $ T $ is a partial order and $ \mu $ is a probability measure on $ T $
with $ \open{S}_i $, $ \closed{S}_i $, $ \open{N}_i $, and $ \closed{N}_i $ defined as in the introduction.
Let $ b, n \geq 0 $ and $ i \geq 1 $ be integers then
\begin{equation} \label{eq:prob-limit-slln} \begin{split}
\lim_{b \to \infty} \frac{\# \{j \in [0, b) \;:\; \open{S}_j, \open{N}_j = n\}}{\# \{j \in [0, b) \;:\; \open{S}_j\}}
&= \probWhen{\open{N}_i = n}{\open{S}_i} \\
\lim_{b \to \infty} \frac{\# \{j \in [0, b) \;:\; \closed{S}_j, \closed{N}_j = n\}}{\# \{j \in [0, b) \;:\; \closed{S}_j\}}
&= \probWhen{\closed{N}_i = n}{\closed{S}_i}
\end{split} \end{equation}
almost surely.
\end{thm}

\subsection{Rearrangement Insensitivity}
\label{sub:rear-insens}

Suppose we define the two element total order $ T_{xy} := \{x, y\} $ where $ x < y $.
For some $ p \in (0, 1) $, we define the probability measure $ \mu_{xy} $ on $ T_{xy} $
where $ \mu_{xy}(\{x\}) = p $ and $ \mu_{xy}(\{y\}) = 1 - p $.
Similarly, we may define $ T_{yx} = \{x, y\} $ as the partial order where $ y < x $
and we endow it with the probability measure $ \mu_{yx} $ on $ T_{yx} $
where $ \mu_{yx}(\{x\}) = \mu_{xy}(\{x\}) = p $ and $ \mu_{yx}(\{y\}) = \mu_{xy}(\{y\}) = 1 - p $.
Because the number of atoms of each mass is the same for $ \mu_{xy} $ and $ \mu_{yx} $,
we can think of $ \mu_{xy} $ as a rearrangement of $ \mu_{yx} $.
More generally, for any finite non-negative measures $ \mu $ and $ \mu' $
on total orders $ T $ and $ T' $, respectively,
we say $ \mu' $ is a \textbf{rearrangement} of $ \mu $
if there exists a collection of measures on total orders $ \{\langle \mu_\gamma, T_\gamma \rangle\}_{\gamma \in \Gamma} $
and a permutation $ \sigma $ over $ \Gamma $
such that $ \mu = \bigotimes_\gamma \mu_\gamma $ and $ \mu' = \bigotimes_\gamma \mu_{\sigma(\gamma)} $.

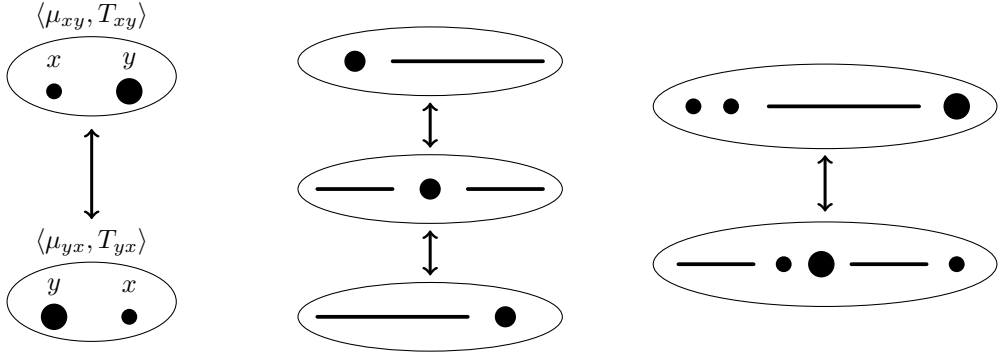
\begin{figure}[H] \label{fig:series-parallel} \centering
\begin{tikzpicture}
\node at (-2.5, 4) {$ \langle \mu_{xy}, T_{xy} \rangle $};
\fill (-3, 3) circle (3pt);
\node at (-3, 3.4) {$ x $};
\fill (-2, 3) circle (5pt);
\node at (-2, 3.4) {$ y $};
\draw (-2.5, 3.2) ellipse [x radius=32pt, y radius=15pt];

\draw[line width=0.4mm, <->] (-2.5, 2.5) -- (-2.5, 1.3);

\node at (-2.5, 1) {$ \langle \mu_{yx}, T_{yx} \rangle $};
\fill (-3, 0) circle (5pt);
\node at (-3, 0.4) {$ y $};
\fill (-2, 0) circle (3pt);
\node at (-2, 0.4) {$ x $};
\draw (-2.5, 0.2) ellipse [x radius = 32pt, y radius=15pt];

\fill (1, 3.4) circle (4pt);
\draw[line width=0.5mm, line cap=round] (1.5, 3.4) -- (3.5, 3.4);
\draw (2, 3.4) ellipse [x radius=50pt, y radius=13pt];

\draw[line width=0.4mm, <->] (2, 2.85) -- (2, 2.25);

\draw[line width=0.5mm, line cap=round] (0.5, 1.7) -- (1.5, 1.7);
\fill (2, 1.7) circle (4pt);
\draw[line width=0.5mm, line cap=round] (2.5, 1.7) -- (3.5, 1.7);
\draw (2, 1.7) ellipse [x radius=50pt, y radius=13pt];

\draw[line width=0.4mm, <->] (2, 1.15) -- (2, 0.55);

\draw[line width=0.5mm, line cap=round] (0.5, 0) -- (2.5, 0);
\fill (3, 0) circle (4pt);
\draw (2, 0) ellipse [x radius=50pt, y radius=13pt];

\fill (5.5, 2.8) circle (3pt);
\fill (6, 2.8) circle (3pt);
\draw[line width=0.5mm, line cap=round] (6.5, 2.8) -- (8.5, 2.8);
\fill (9, 2.8) circle (5pt);
\draw (7.25, 2.8) ellipse [x radius=65pt, y radius=16pt];

\draw[line width=0.4mm, <->] (7.25, 2.15) -- (7.25, 1.4);

\draw[line width=0.5mm, line cap=round] (5.3, 0.7) -- (6.3, 0.7);
\fill (6.7, 0.7) circle (3pt);
\fill (7.2, 0.7) circle (5pt);
\draw[line width=0.5mm, line cap=round] (7.6, 0.7) -- (8.6, 0.7);
\fill (9, 0.7) circle (3 pt);
\draw (7.25, 0.7) ellipse [x radius=65pt, y radius=16pt];

\end{tikzpicture}
\caption{Three examples of rearranged total orders.
Circles represent atoms and lines represent diffuse mass.}
\end{figure}

We are interested in the distribution of run lengths for $ \mu_{xy} $ and $ \mu_{yx} $.
Both $ \mu_{yx} $ and $ \mu_{xy} $ have one atom of mass $ p $ and one of mass $ 1 - p $.
Thus, by \cref{prop:total-run-funcs}, the run functions of $ \mu_{xy} $ and $ \mu_{yx} $ will be
\[
\open{P}_{\mu_{xy}}(Z) = \left(1 + \frac{Z}{3}\right) \left(1 + \frac{2 Z}{3}\right) = \open{P}_{\mu_{yx}}(Z)
\quad \text{ and } \quad
\closed{P}_{\mu_{xy}}(Z) = \frac{1}{1 - \frac{Z}{3}} \cdot \frac{1}{1 - \frac{2 Z}{3}} = \closed{P}_{\mu_{yx}}(Z)
\]
which with \cref{thm:length-pgf} means that the PGFs of the run lengths agree between $ \mu_{xy} $ and $ \mu_{yx} $.
In particular, the distributions of the run lengths for $ \mu_{xy} $ and $ \mu_{yx} $ will be the same.
Therefore, we might say the events $ \{\open{N}_i = n\} $ and $ \{\closed{N}_i = n\} $ are insensitive to rearrangement.
More generally, we say an event $ E \in \mathcal{F} $ associated with the random sequence $ X_0, X_1, X_2, \ldots $
is \textbf{rearrangement insensitive} if for any probability measures $ \mu $ and $ \mu' $
which are rearrangements of each other $ \prob[\mu]{E} = \prob[\mu']{E} $
where $ \mathbb{P}_\mu $ is the probability of an event when the distribution of $ X_i $ is $ \mu $.
\medskip

\begin{prop} \label{prop:runs-insens}
For any $ n \geq 0 $ and any $ i \geq 0 $,
the events $ \{\open{N}_i = n\} $ and $ \{\closed{N}_i = n\} $
are insensitive to rearrangement.
\end{prop}
\begin{proof}
Suppose $ T $ and $ T' $ are total orders with finite non-negative measures $ \mu $ and $ \mu' $.
Also, assume that $ \mu $ and $ \mu' $ are rearrangements.
If $ \mathcal{A} $ and $ \mathcal{A}' $ are their sets of atoms
then there must exist a bijection $ \phi : \mathcal{A} \to \mathcal{A}' $.
Further, $ \mu(\{\alpha\}) = \mu'(\{\phi(\alpha)\})$ for any $ \alpha \in \mathcal{A} $.
Thus, the weights of the atoms are the same and
\[
\mu(T \setminus \mathcal{A}) = \mu(T) - \sum_{\alpha \in \mathcal{A}} \mu(\{\alpha\})
= \mu(T) - \sum_{\alpha \in \mathcal{A}} \mu(\{\phi(\alpha)\})
= \mu'(T') - \sum_{\alpha' \in \mathcal{A}'} \mu'(\{\alpha'\})
= \mu'(T' \setminus \mathcal{A}')
\]
so their diffuse masses are the same.
Therefore, by \cref{prop:total-run-funcs}, we know $ \open{P}_\mu(Z) = \open{P}_{\mu'}(Z) $ and $ \closed{P}_\mu(Z) = \closed{P}_\mu(Z) $.
Then, by \cref{thm:length-pgf} and \cref{coro:length-stats}, the PGFs of $ \open{N}_i $ and $ \closed{N}_i $
will be the same for $ \mu $ and $ \mu' $.
Because the PGFs determine the probabilities of $ \{\open{N}_i = n\} $ and $ \{\closed{N}_i = n\} $, we know
\[
\prob[\mu]{\open{N}_i = n} = \prob[\mu']{\open{N}_i = n}
\quad \text{ and } \quad
\prob[\mu]{\closed{N}_i = n} = \prob[\mu']{\closed{N}_i = n}.
\]
\end{proof}

Rearrangement insensitivity may not seem like an especially unusual property.
However, most events that one encounters are \emph{sensitive} to rearrangement.
In the following proposition, we will look at a class of examples of rearrangement sensitive events.

\begin{prop} \label{prop:record-sens}
Suppose $ \mu $ is a probability measure on a total order $ T $
and $ \{X_i\}_{i=0}^\infty $ is a random sequence in $ T $
independently and identically distributed as $ \mu $.
Then, for any $ n \geq 2 $, the event
\[ R_n := \{X_n \geq \max(X_0, X_1, \ldots, X_{n-1})\} \]
is sensitive to rearrangement.
\end{prop}
\begin{proof}
It suffices to consider a single counter-example.
We will use $ \mu_{xy} $ and $ \mu_{yx} $ from above.
Now, for any $ n \geq 2 $, using independence of $ X_0, \ldots, X_{n-1} $,
\[ \begin{split}
\prob[\mu_{xy}]{R_n} &= \prob[\mu_{xy}]{X_n \geq X_0, X_n \geq X_1, \ldots, X_n \geq X_{n-1}} \\
&= p \cdot \probWhen[\mu_{xy}]{X_n \geq X_0, \ldots, X_n \geq X_{n-1}}{X_n = x}
+ (1 - p) \cdot \probWhen[\mu_{yx}]{X_n \geq X_0, \ldots, X_n \geq X_{n-1}}{X_n = y} \\
&= p \cdot \prob[\mu_{xy}]{x \geq X_0} \cdots \prob[\mu_{xy}]{x \geq X_{n-1}}
+ (1 - p) \cdot \prob[\mu_{xy}]{y \geq X_0} \cdots \prob[\mu_{xy}]{y \geq X_{n-1}} \\
&= p \cdot p^n + (1 - p) \cdot 1^n = p^{n+1} + (1 - p).
\end{split} \]
By a similar calculation, one concludes that $ \prob[\mu_{yx}]{R_n} = (1 - p)^{n+1} + p $.
Then, setting $ p = \frac{1}{3} $ and taking their difference
\[
\prob[\mu_{xy}]{R_n} - \prob[\mu_{yx}]{R_n} = \left(\frac{1}{3^{n+1}} + \frac{2}{3}\right) - \left(\frac{2^{n+1}}{3^{n+1}} + \frac{1}{3}\right)
= \frac{1}{3} + \frac{1 - 2^{n+1}}{3^{n+1}} = \frac{3^n - 2^{n+1} + 1}{3^{n+1}}.
\]
For all $ n \geq 2 $, $ 3^n > 2^{n+1} $ so $ 3^n - 2^{n+1} + 1 > 0 $ implying $ \prob[\mu_{xy}]{R_n} \neq \prob[\mu_{yx}]{R_n} $.
\end{proof}

The choice of $ \mu_{xy} $ and $ \mu_{yx} $ here is not unique.
One can check that most measures that have non-trivial rearrangements will exhibit rearrangement sensitivity in $ R_n $.
It is also not difficult to find other events which are sensitive to rearrangement.
Though, the condition for whether an event is sensitive and insensitive appears rather subtle.
Notice from \cref{prop:runs-insens} that $ \{\closed{N}_0 = 2\} = \{X_0 \leq X_1 > X_2\} $ and $ \{\closed{N}_0 \geq 3\} = \{X_0 \leq X_1 \leq X_2\} $ are insensitive.
However, from \cref{prop:record-sens}, we know $ \prob{R_2} = \prob{X_2 \geq X_0, X_2 \geq X_1} = \prob{X_0 \leq X_2 \geq X_1} = \prob{X_0 \leq X_1 \geq X_2} $ is sensitive.
Thus, $ \{X_0 \leq X_1 \geq X_2\} $ is sensitive, but if we flip the second inequality or remove the equality from it
then we get the insensitive events $ \{X_0 \leq X_1 \leq X_2\} $ and $ \{X_0 \leq X_1 > X_2\} $.

\section{Examples}
\label{sect:examples}

\subsection{Radioactive Decay}
Here, the total order is $ T = \mathbb{R}_{\geq 0} $ and $ \mu $ is exponential
with $ \mu([0, t)) = 1 - e^{-\lambda t} $ for some $ \lambda > 0 $.
However, as a result of \cref{lem:total-diffuse-func},
it is only important that $ \mu $ has no atoms and $ T $ is a total order.
By \cref{thm:total-length-pgf}, the PGFs are
\[
G_{\open{N}_0}(Z) = G_{\closed{N}_0}(Z) = \frac{1 + (Z - 1) e^Z}{Z}
\quad \text{ and when } i \geq 1, \;
G_{\open{N}_i}(Z) = G_{\closed{N}_i}(Z) = 1 + \frac{2 (Z - 1) + 2 (Z - 1)^2 e^Z}{Z^2}.
\]
Using \cref{coro:total-length-stats}, the means are
\[
\expect{\open{N}_0} = \expect{\closed{N}_0} = e - 1
\quad \text{ and when } i \geq 1, \;
\expectWhen{\open{N}_i}{\open{S}_i} = \expectWhen{\closed{N}_i}{\closed{S}_i} = 2
\]
and the variances are
\[
\var{\open{N}_0} = \var{\closed{N}_0} = e (3 - e)
\quad \text{ and when } i \geq 1, \;
\varWhen{\open{N}_i}{\open{S}_i} = \varWhen{\closed{N}_i}{\closed{S}_i} = 4e - 10.
\]

\subsection{N-Sided Die}
Recall that $ T = \{1, \ldots, n\} $ and $ \mu(\{k\}) = \frac{1}{n} $ for all $ k \in [1, n] $.
Therefore, the atoms will be $ \mathcal{A} = \{1, \ldots, n\} $ with $ m_\alpha = \frac{1}{n} $ for all $ \alpha \in \mathcal{A} $
and there is no diffuse mass so $ m_d = 0 $.
From \cref{thm:total-length-pgf}, we know the PGFs are
\[
G_{\open{N}_0}(Z) = \frac{1 + (Z - 1) \left(1 + \frac{Z}{n}\right)^n}{Z}
\quad \text{ and } \quad
G_{\closed{N}_0}(Z) = \frac{1 + \frac{Z - 1}{(1 - Z / n)^n}}{Z}
\]
\[
G_{\open{N}_i}(Z) = 1 + \frac{2 n (Z - 1)}{(n + 1) Z} G_{\open{N}_0}(Z)
\quad \text{ and } \quad
G_{\closed{N}_i}(Z) = 1 + \frac{2 n (Z - 1)}{(n - 1) Z} G_{\closed{N}_0}(Z).
\]
Using \cref{coro:total-length-stats}, we can say
\[
\expect{\open{N}_0} = \left(1 + \frac{1}{n}\right)^n - 1
\quad \text{ and } \quad
\expect{\closed{N}_0} = \left(\frac{1}{1 - \frac{1}{n}}\right)^n - 1 = \left(1 + \frac{1}{n - 1}\right)^n - 1
\]
as well as for $ i \geq 1 $,
\[
\expectWhen{\open{N}_i}{\open{S}_i} = \frac{2}{1 + n \frac{1}{n^2}} = \frac{2n}{n + 1} = 2 - \frac{2}{n + 1}
\quad \text{ and } \quad
\expectWhen{\closed{N}_i}{\closed{S}_i} = \frac{2}{1 - n \frac{1}{n^2}} = \frac{2n}{n - 1} = 2 + \frac{2}{n - 1}.
\]
Notice that in all cases, as $ n \to \infty $, the expected values approaches those
for a diffuse distribution: $ \expect{\open{N}_0}, \expect{\closed{N}_0} \to e - 1 $
and $ \expectWhen{\open{N}_i}{\open{S}_i}, \expectWhen{\closed{N}_i}{\closed{S}_i} \to 2 $.
This reflects the fact that a diffuse distribution (e.g.\ the uniform distribution on $ [0, 1] $) can be approximated by separating it into very small equal slices.
\medskip

Next, the run functions are $ \open{P}_\mu(Z) = \left(1 + \frac{Z}{n}\right)^n $ and $ \closed{P}_\mu(Z) = \left(\frac{1}{1 - Z / n}\right)^n $
and their derivatives are
\[ \begin{split}
\open{P}_\mu'(Z) &= \frac{d}{dZ} \left(1 + \frac{Z}{n}\right)^n = \frac{n}{n} \left(1 + \frac{Z}{n}\right)^{n-1} \quad \text{ and} \\
\closed{P}_\mu'(Z) &= \frac{d}{dZ} \left(\frac{1}{1 - Z / n}\right)^n = \frac{d}{dZ} \left(\frac{n}{n - Z} \right)^n
= n \frac{n}{(n - Z)^2} \left(\frac{n}{n - Z}\right)^{n-1} = \left(\frac{n}{n - Z}\right)^{n+1}
\end{split} \]
implying that $ \open{P}_\mu'(1) = \left(1 + \frac{1}{n}\right)^{n-1} $ and $ \closed{P}_\mu'(1) = \left(1 + \frac{1}{n-1}\right)^{n+1} $.
We define $ a := 1 + \frac{1}{n} $ and $ b := 1 + \frac{1}{n-1} $.
Using \cref{coro:length-stats}, the variances for $ N_0 $ are
\[
\var{\open{N}_0} = \open{P}_\mu(1) - \open{P}_\mu^2(1) + 2 \open{P}_\mu'(1) = a^n - a^{2n} + 2 a^{n-1}
\quad \text{ and } \quad
\var{\closed{N}_0} = \closed{P}_\mu(1) - \closed{P}_\mu^2(1) + 2 \closed{P}_\mu'(1) = b^n - b^{2n} + 2 b^{n+1}.
\]
Also when $ i \geq 1 $, the variances are
\[ \begin{split}
\varWhen{\open{N}_i}{\open{S}_i} &= \frac{2}{1 + n \frac{1}{n^2}} \left(-3 + 2 \left(1 + \frac{1}{n}\right)^n - \frac{2}{1 + n \frac{1}{n^2}}\right)
= \frac{2}{a} \left(-3 + 2 a^n - \frac{2}{a}\right) = 4 a^{n-1} - \frac{6}{a} - \frac{4}{a^2}
\quad \text{ and} \\
\varWhen{\closed{N}_i}{\closed{S}_i} &= \frac{2}{1 - n \frac{1}{n^2}} \left(-3 + 2 \left(1 + \frac{1}{n-1}\right)^n - \frac{2}{1 - n \frac{1}{n^2}}\right)
= 2b \left(-3 + 2 b^n - 2b\right) = 4 b^{n+1} - 6b - 4b^2.
\end{split} \]

\subsection{Dart Board with Bullseye}
Here, $ T = (0, 1] $ with $ \mu(\{1\}) = p $ and $ \mu((0, 1)) = 1 - p $ with the mass distributed uniformly.
Thus, $ \mu $ has one atom $ \mathcal{A} = \{1\} $ and diffuse mass $ m_d = 1 - p $.
Using \cref{thm:total-length-pgf}, the PGFs will be
\[ \def\arraystretch{2.2} \begin{array}{clclc}
& G_{\open{N}_0}(Z) = \dfrac{1 + (Z - 1) (1 + pZ) e^{(1 - p) Z}}{Z} & \text{and} &
G_{\closed{N}_0}(Z) = \dfrac{1 + \frac{Z - 1}{1 - pZ} e^{(1 - p) Z}}{Z} &
\text{as well as} \\
\text{for } i \geq 1, &
G_{\open{N}_i}(Z) = 1 + \dfrac{2 (Z - 1)}{Z (1 + p^2)} G_{\open{N}_0}(Z) & \text{and} &
G_{\closed{N}_i}(Z) = 1 + \dfrac{2 (Z - 1)}{Z (1 - p^2)} G_{\closed{N}_0}(Z). &
\end{array} \]
From \cref{coro:total-length-stats}, the expected run lengths are
\[ \def\arraystretch{2.2} \begin{array}{clclc}
& \expectWhen{\open{N}_0}{\open{S}_0} = e^{1 - p} (1 + p) - 1 & \text{and} &
\expectWhen{\closed{N}_0}{\closed{S}_0} = \dfrac{e^{1 - p}}{1 - p} - 1 &
\text{as well as} \\
\text{for } i \geq 1, &
\expectWhen{\open{N}_i}{\open{S}_i} = \dfrac{2}{1 + p^2} & \text{and} &
\expectWhen{\closed{N}_i}{\closed{S}_i} = \dfrac{2}{1 - p^2}. &
\end{array} \]
For the variances, we have
\[ \def\arraystretch{2.2} \begin{array}{clclc}
& \var{\open{N}_0} = e^{1 - p} (1 + p) \left(3 - e^{1 - p} (1 + p) - \dfrac{2 p^2}{1 + p} \right) & \text{and} &
\var{\closed{N}_0} = \dfrac{e^{1 - p}}{1 - p} \left(3 - \dfrac{e^{1 - p}}{1 - p} + \dfrac{2 p^2}{1 - p} \right)
\; \text{as well as} \\
\text{for } i \geq 1, &
\varWhen{\open{N}_i}{\open{S}_i} = \dfrac{2}{1 + p^2} \left( -3 + 2 e^{1 - p} (1 + p) - \dfrac{2}{1 + p^2} \right) & \text{and} &
\varWhen{\closed{N}_i}{\closed{S}_i} = \dfrac{2}{1 - p^2} \left( -3 + 2 \dfrac{e^{1 - p}}{1 - p} - \dfrac{2}{1 - p^2} \right).
\end{array} \]

\subsection{Faulty Light Bulb}
In the light bulb example, we recall the order is $ T = \mathbb{N}_{>0} $
and $ \mu $ is a geometric distribution with $ \mu(\{k\}) = p (1 - p)^{k-1} $ for $ k > 0 $.
Because it has the same order theoretic properties, we will consider $ T = \mathbb{N} $
with $ \mu(\{k\}) = p (1 - p)^k $ for $ k \geq 0 $.
Then, $ \mu $ has atoms $ \mathcal{A} = \{0, 1, 2, \ldots\} $ and no diffuse mass.
From \cref{prop:total-run-funcs}, we know for any $ z \in \mathbb{C} \setminus \{0\} $,
the run functions will satisfy
\[
\open{P}_\mu(z) = \prod_{k=0}^\infty \left( \big. 1 + p (1 - p)^k z \right)
\quad \text{ and when } |z| \leq 1, \quad
\closed{P}_\mu(z) = \left(\prod_{k=0}^\infty \left( \big. 1 - p (1 - p)^k z \right)\right)^{-1}.
\]
Here, we observe that these functions are related to the well-known $ q $-Pochhammmer symbol $ (a; q)_n $ which is defined to be
\[ (a; q)_n = \prod_{k=0}^{n-1} (1 - a q^k) \]
for all $ n > 0 $.
This function plays a significant role in the theory of $ q $-series (see \cite{GAnd} for more information).
We can then say $ \open{P}_\mu(z) = (-pz; 1 - p)_\infty $ and $ \closed{P}_\mu(z) = 1 / (pz; 1 - p)_\infty $.
We can also calculate
\[ \sum_{k=0}^\infty m_k^2 = \sum_{k=0}^\infty p^2 (1 - p)^{2k} = \frac{p^2}{1 - (1 - p)^2} = \frac{p}{2 - p}. \]
Then from \cref{thm:length-pgf}, we get $ G_{\open{N}_0}(z) = \frac{1}{z} \left( \big. 1 + (z - 1) (-pz; 1 - p)_\infty \right) $
and when $ |z| \leq 1 $, $ G_{\closed{N}}(z) = \frac{1}{z} \left( \big. 1 + \frac{z - 1}{(pz; 1 - p)_\infty} \right) $.
Further, for $ i \geq 1 $, we have
\[ \begin{split}
G_{\open{N}_i}(z) &= 1 + \frac{2 (z - 1)}{z \left(1 + \frac{p}{2 - p} \right)} G_{\open{N}_0}(z) 
= 1 + \frac{(2 - p)(z - 1)}{z} G_{\open{N}_0}(z) \\
\text{ and when } |z| \leq 1, \quad
G_{\closed{N}_i}(z) &= 1 + \frac{2 (z - 1)}{z \left(1 - \frac{p}{2 - p}\right)} G_{\closed{N}_0}(z)
= 1 + \frac{(2 - p)(z - 1)}{z (1 - p)} G_{\closed{N}_0}(z).
\end{split} \]
From \cref{coro:length-stats} or equivalently \cref{coro:total-length-stats}, we get
\[ \def\arraystretch{2.2} \begin{array}{clclc}
& \expect{\open{N}_0} = (-p; 1 - p)_\infty - 1 & \text{and} &
\expect{\closed{N}_0} = \dfrac{1}{(p; 1 - p)_\infty} - 1 &
\text{as well as} \\
\text{for } i \geq 1, &
\expectWhen{\open{N}_i}{\open{S}_i} = \dfrac{2}{1 + \frac{p}{2 - p}} = 2 - p & \text{and} &
\expectWhen{\closed{N}_i}{\closed{S}_i} = \dfrac{2}{1 - \frac{p}{2 - p}} = \dfrac{2 - p}{1 - p}. &
\end{array} \]
Further, the variances are
\[ \begin{split}
\var{\open{N}_0} &= (-p; 1 - p)_\infty \left(3 - (-p; 1 - p)_\infty - 2 \sum_{k=0}^\infty \frac{p^2 (1 - p)^{2k}}{1 + p (1 - p)^k} \right)
\quad \text{ and} \\
\var{\closed{N}_0} &= \frac{1}{(p; 1 - p)_\infty} \left(3 - \frac{1}{(p; 1 - p)_\infty} + 2 \sum_{k=0}^\infty \frac{p^2 (1 - p)^{2k}}{1 - p(1 - p)^k} \right)
\end{split} \]
as well as for $ i \geq 1 $,
\[ \begin{split}
\varWhen{\open{N}_i}{\open{S}_i} &= (2 - p) \left(-3 + 2 (-p; 1 - p)_\infty - (2 - p) \Big. \right)
= (2 - p) \left(p - 5 + 2 (-p; 1 - p)_\infty \Big. \right)
\quad \text{ and} \\
\varWhen{\closed{N}_i}{\closed{S}_i} &= \frac{2 - p}{1 - p} \left(-3 + \frac{2}{(p; 1 - p)_\infty} - \frac{2 - p}{1 - p}\right)
= \frac{2 - p}{1 - p} \left(\frac{4p - 5}{1 - p} + \frac{2}{(p; 1 - p)_\infty} \right).
\end{split} \]

\subsection{Even-Sided Die as Partial Order}
The partial order $ T $ in this case is $ \{1, \ldots, 2n\} $ where $ \{1, 3, \ldots, 2n - 1\} $ and $ \{2, 4, \ldots, 2n\} $ are chains
and the probability measure $ \mu $ is $ \mu(\{k\}) = \frac{1}{2n} $ for all $ k \in [1, 2n] $.
We say $ \mu_{\mathrm{odd}} $ and $ \mu_{\mathrm{even}} $ are the restrictions of $ \mu $ to $ \{1, 3, \ldots, 2n - 1\} $ and $ \{2, 4, \ldots, 2n\} $, respectively.
Then, $ \mu = \mu_{\mathrm{odd}} \oplus \mu_{\mathrm{even}} $.
By \cref{eq:total-run-funcs}, the run functions for $ \mu_{\mathrm{odd}} $ and $ \mu_{\mathrm{even}} $ are
\[
\open{P}_{\mu_{\mathrm{odd}}}(Z) = \open{P}_{\mu_{\mathrm{even}}}(Z) = \left(1 + \frac{Z}{2n} \right)^n
\quad \text{ and } \quad
\closed{P}_{\mu_{\mathrm{odd}}}(Z) = \closed{P}_{\mu_{\mathrm{even}}}(Z) = \left(\frac{1}{1 - Z / (2n)}\right)^n = \left(\frac{2n}{2n - Z}\right)^n.
\]
Thus, by \cref{thm:series-parallel},
\[ \begin{matrix}
\open{P}_\mu(Z) = \open{P}_{\mu_{\mathrm{odd}}}(Z) + \open{P}_{\mu_{\mathrm{even}}}(Z) - 1 = 2 \left(1 + \frac{Z}{2n} \right)^n - 1 &
\closed{P}_\mu(Z) = \closed{P}_{\mu_{\mathrm{odd}}}(Z) + \closed{P}_{\mu_{\mathrm{even}}}(Z) - 1 = 2 \left(\frac{2n}{2n - Z}\right)^n - 1 \\
\open{P}_\mu'(Z) = \frac{2n}{2n} \left(1 + \frac{Z}{2n}\right)^{n-1} &
\closed{P}_\mu'(Z) = 2n \frac{2n}{(2n - Z)^2} \left(\frac{2n}{2n - Z}\right)^{n-1} = \left(\frac{2n}{2n - Z}\right)^{n+1} \\
\open{P}_\mu''(Z) = \frac{n - 1}{2n} \left(1 + \frac{Z}{2n}\right)^{n-2} &
\closed{P}_\mu''(Z) = (n + 1) \frac{2n}{(2n - Z)^2} \left(\frac{2n}{2n - Z}\right)^n = \frac{n + 1}{2n} \left(\frac{2n}{2n - Z}\right)^{n + 2}.
\end{matrix} \]
Therefore, using \cref{coro:length-stats}, the expected values for $ N_0 $ are
\[
\expect{\open{N}_0} = \open{P}_\mu(1) - 1 = 2 \left(1 + \frac{1}{2n}\right)^n - 2
\quad \text{ and } \quad
\expect{\closed{N}_0} = \closed{P}_\mu(1) - 1 = 2 \left(1 + \frac{1}{2n - 1}\right)^n - 2.
\]
The expected values for $ N_i $ when $ i \geq 1 $ are
\[
\expectWhen{\open{N}_i}{\open{S}_i = 1} = \frac{2}{2 - \open{P}_\mu''(0)} = \frac{2}{2 - \frac{n - 1}{2n}} = \frac{4n}{3n + 1}
\quad \text{ and } \quad
\expectWhen{\closed{N}_i}{\closed{S}_i = 1} = \frac{2}{2 - \closed{P}_\mu''(0)} = \frac{2}{2 - \frac{n + 1}{2n}} = \frac{4n}{3n - 1}.
\]
Using \cref{thm:length-pgf} and \cref{coro:length-stats} with $ \open{P}_\mu(Z) $, $ \closed{P}_\mu(Z) $, $ \open{P}_\mu'(Z) $, and $ \closed{P}_\mu'(Z) $,
one can also derive the PGFs and variances for $ \open{N}_i $ and $ \closed{N}_i $.

\section{Outline of Proofs}
For the sake of concision, we will omit the $ \bullet $ and $ \circ $ in equations
when the results apply to both the strict and non-strict cases.
Since we will often work with subsets of $ T^n $,
we may write set schema as
\[ \{a_1 < a_2, a_3 < a_4\}_{T^n} := \{\langle a_1, \ldots, a_n \rangle \in T^n \;:\; a_1 < a_2, a_3 < a_4\}. \]
As above, we will assume that $ \{a_1 < a_2\} \subseteq T^2 $
is always measurable with respect to the Borel $ \sigma $-algebra.
Because of the value of addressing orders abstractly,
we will prove the some of the results of \cref{sub:series-parallel} before addressing
the totally ordered case.
Here is an outline of the proofs.
\begin{itemize}
	\item[] \textbf{Statistics of Run Lengths from Run Functions:}
			\textsl{Proofs of \cref{thm:length-pgf} and \cref{coro:length-stats}}
	\begin{itemize}
		\item In \cref{lem:func-conv-rad}, we show that the run functions are convergent in a sufficiently large radius.
		\item We establish the relationship between the run functions and the PGFs of $ \open{N}_i $ and $ \closed{N}_i $
		proving \cref{thm:length-pgf}.
		\item Using the formulae for the PGFs, we derive the means and variances of $ \open{N}_i $ and $ \closed{N}_i $
		proving \cref{coro:length-stats}.
	\end{itemize}

	\item[] \textbf{Law of Large Numbers:}
			\textsl{Proof of \cref{thm:slln}} \newline
	For a given realization of a random sequence, we demonstrate that
	the proportion of runs of length $ n $ almost surely approaches
	the distribution given by \cref{eq:length-pgf}.

	\item[] \textbf{Run Functions of Limits of Measures:}
			\textsl{Proof of \cref{prop:func-limit}}
	\begin{itemize}
		\item In \cref{lem:func-diff-bound}, we show that
		the distance between $ P_{\mu_1}(z) $ and $ P_{\mu_2}(z) $
		is bounded by $ \|\mu_1 - \mu_2\| $.
		\item We prove \cref{prop:func-limit} that for any convergent increasing sequence of measures,
		their run functions converge uniformly on bounded sets.
	\end{itemize}

	\item[] \textbf{Combining Run Functions:}
			\textsl{Proof of \cref{prop:comb-funcs}}
	\begin{itemize}
		\item We show that for finite non-negative measures $ \mu_1 $ and $ \mu_2 $,
		the run functions of the series and parallel composition
		are $ P_{\mu_1 \otimes \mu_2}(Z) = P_{\mu_1}(Z) \cdot P_{\mu_2}(Z) $
		and $ P_{\mu_1 \oplus \mu_2}(Z) - 1 = (P_{\mu_1}(Z) - 1) + (P_{\mu_2}(Z) - 1) $.
		\item Using \cref{prop:func-limit}, we extend to countable $ \Gamma $ which proves \cref{prop:comb-funcs}.
	\end{itemize}
	
	\item[] \textbf{Run Lengths in Total Orders \& Countably Series-Parallel Partial Orders:}
			\textsl{Proofs of \cref{thm:total-length-pgf}, \cref{coro:total-length-stats},
			\cref{prop:total-run-funcs}, and \cref{thm:series-parallel}}
	\begin{itemize}
		\item We calculate the run functions for
		a diffuse measure and a degenerate measure on a total order.
		\item Using \cref{prop:comb-funcs}, we describe the run functions
		of an arbitrary measure on a total order
		proving \cref{prop:total-run-funcs} and \cref{thm:total-length-pgf}.
		\item We apply \cref{coro:length-stats} to the run functions
		to obtain the means, variances, and PGFs of run lengths for total orders
		proving \cref{coro:total-length-stats}.
		\item We prove \cref{thm:series-parallel} by showing that a measure
		on a countably series-parallel partial order is equal
		to the series or parallel composition of its restrictions.
	\end{itemize}
\end{itemize}

\section{Statistics of Run Lengths from Run Functions}

\begin{defn}
Suppose $ \mu $ is a finite non-negative measure over a partial order $ T $
and $ \{\langle x, y \rangle \in T^2 \;:\; x < y\} $ is measurable.
Then, for each integer $ n \geq 1 $, we define the \textbf{strict run coefficients}, $ \open{L}_n(\mu) $,
and \textbf{non-strict run coefficients}, $ \closed{L}_n(\mu) $, as
\[
\open{L}_n(\mu) := \begin{cases}
\hfil 1 & \text{ for } n = 0 \\
\hfil \mu(T) & \text{ for } n = 1 \\
\mu^n(\{a_1 < \ldots < a_n\}_{T^n}) & \text{ otherwise}
\end{cases}
\quad \text{ and } \quad
\closed{L}_n(\mu) := \begin{cases}
\hfil 1 & \text{ for } n = 0 \\
\hfil \mu(T) & \text{ for } n = 1 \\
\mu^n(\{a_1 \leq \ldots \leq a_n\}_{T^n}) & \text{ otherwise}
\end{cases}.
\]
When clear from context, we may omit $ \mu $ and simply say $ \open{L}_n $ and $ \closed{L}_n $.
Note that, using Tonelli's theorem, the measurability of $ \{x < y\} $ implies
that $ \{a_1 < \ldots < a_n\} $ and $ \{a_1 \leq \ldots \leq a_n\} $ are measurable for all $ n \geq 2 $.
\end{defn}

Observe that with $ \open{P}_\mu(Z) $ and $ \closed{P}_\mu(Z) $ defined
as in \cref{eq:strict-run-func-defn} and \cref{eq:non-strict-run-func-defn},
\[
\open{P}_\mu(Z) = \sum_{n=0}^\infty \open{L}_n(\mu) Z^n
\quad \text{ and } \quad
\closed{P}_\mu(Z) = \sum_{n=0}^\infty \closed{L}_n(\mu) Z^n.
\]

\begin{lem} \label{lem:func-conv-rad}
For any finite non-negative measure $ \mu $ on a partial order $ T $,
the radius of convergence of $ \open{P}_\mu(Z) $ is infinite
and the radius of convergence of $ \closed{P}_\mu(Z) $ is at least
\[ \sqrt{\frac{2}{\|\mu\| \cdot \left(\|\mu\| + \sup_{\alpha \in \mathcal{A}} \mu(\{\alpha\}) \right)}} \geq \frac{1}{\|\mu\|} \]
where equality occurs only if $ \mu $ is degenerate.
\end{lem}
\begin{proof}
First, we consider the strict case.
Let $ n \geq 2 $ and let $ \Sigma_n $ be the set of permutations of $ \{1, \ldots, n\} $.
For each $ \sigma \in \Sigma_n $, we define
\[ A_\sigma := \{a_{\sigma(1)} < a_{\sigma(2)} < \ldots < a_{\sigma(n)}\}_{T^n}. \]
Now, if $ \sigma, \sigma' \in \Sigma_n $ with $ \sigma \neq \sigma' $ then $ A_\sigma \cap A_{\sigma'} = \emptyset $,
because if $ a_{\sigma(1)}, \ldots, a_{\sigma(n)} $ is an ascending sequence then $ a_{\sigma'(1)}, \ldots, a_{\sigma'(n)} $ will not be.
Therefore, using symmetry
\[
n! \cdot \open{L}_n = n! \cdot \mu^n(\{a_1 < \ldots < a_n\}_{T^n})
= \sum_{\sigma \in \Sigma_n} \mu^n(A_\sigma)
= \mu^n\left(\bigcup_{\sigma \in \Sigma_n} A_\sigma \right) \leq \mu^n(T^n) = \|\mu\|^n
\]
which implies $ \open{L}_n \leq \frac{\|\mu\|^n}{n!} $.
Thus, for any $ z \in \mathbb{C} $,
\[
\open{P}_\mu(|z|) = \sum_{n=0}^\infty \open{L}_n(\mu) |z|^n
\leq \sum_{n=0}^\infty \frac{\|\mu\|^n \cdot |z|^n}{n!} = e^{\|\mu\| \cdot |z|} < \infty
\]
so $ \open{P}_\mu(Z) $ has an infinite radius of convergence.
\bigskip

Second, we consider the non-strict case.
Let $ n \geq 4 $ be an integer then
\[ \{a_1 \leq a_2 \leq a_3 \leq \ldots \leq a_n\}_{T^n} \subseteq \{a_1 \leq a_2\}_{T^2} \times \{a_3 \leq \ldots \leq a_n\}_{T^{n-2}} \]
\[
\closed{L}_n = \mu^n(\{a_1 \leq \ldots \leq a_n\}_{T^n})
\leq \mu^2(\{a_1 \leq a_2\}_{T^2}) \cdot \mu^{n-2}(\{a_3 \leq \ldots \leq a_n\}_{T^{n-2}}) = \closed{L}_2 \cdot \closed{L}_{n-2}.
\]
Similarly for $ n = 3 $, we get
\[ \{a_1 \leq a_2 \leq a_3\}_{T^3} \subseteq \{a_1 \leq a_2\}_{T^2} \times T \]
\[ \closed{L}_3 = \mu^3(\{a_1 \leq a_2 \leq a_3\}_{T^3}) \leq \mu^2(\{a_1 \leq a_2\}_{T^2}) \cdot \mu(T) = \closed{L}_2 \cdot \closed{L}_1. \]
Applying this inductively, we get that for all $ n \geq 1 $,
\begin{equation} \label{eq:coef-bounds}
\closed{L}_{2n} \leq \closed{L}_2^n
\quad \text{ and } \quad
\closed{L}_{2n+1} \leq \closed{L}_2^n \cdot \closed{L}_1.
\end{equation}
Let $ x \in S $ then by definition $ \mu(\{x\}) \leq C := \sup_{\alpha \in \mathcal{A}} \mu(\{\alpha\}) $ and so
\[ \mu^2(\{a_1 = a_2\}_{T^2}) = \int \mu(\{a_1 = x\}_T) \dr\mu(x) = \int \mu(\{x\}) \dr\mu(x) \leq \int C \dr\mu(x) = C \|\mu\|. \]
Since $ a_1 \leq a_2 $ and $ a_1 > a_2 $  are disjoint possibilities, we know
\[ \mu^2(\{a_1 \leq a_2\}_{T^2}) + \mu^2(\{a_1 > a_2\}_{T^2}) \leq \|\mu\|^2 \]
and by symmetry $ \mu^2(\{a_1 > a_2\}_{T^2}) = \mu^2(\{a_1 < a_2\}_{T^2}) $ so
\[ 2 \closed{L}_2 = \mu^2(\{a_1 \leq a_2\}_{T^2}) + \mu^2(\{a_1 = a_2\}) + \mu^2(\{a_1 < a_2\}_{T^2}) \leq \|\mu\|^2 + C \|\mu\|. \]
Thus, we have the bound $ \closed{L}_2 \leq \|\mu\| \cdot (\|\mu\| + C) / 2 $.
Let $ z \in \mathbb{C} $ with $ |z| < 1 / \sqrt{\closed{L}_2} $.
Then, using \cref{eq:coef-bounds} and the fact that $ \closed{L}_1 = \|\mu\| > 0 $, we know
\[
\lim_{k \to \infty} \sqrt[2k]{\abs{\closed{L}_{2k} z^{2k}}} \leq \lim_{k \to \infty} \sqrt[2k]{\abs{\closed{L}_2^k z^{2k}}} = |z| \sqrt{\closed{L}_2} \quad \text{ and } \quad
\lim_{k \to \infty} \sqrt[2k+1]{\abs{\closed{L}_{2k+1} z^{2k+1}}} \leq \lim_{k \to \infty} \sqrt[2k]{\abs{\closed{L}_1 \cdot \closed{L}_2^k z^{2k}}} = |z| \sqrt{\closed{L}_2}.
\]
Because this holds for both evens and odds, we know $ \sqrt[k]{\closed{L}_k |z|^k} \to |z| \sqrt{\closed{L}_2} < 1 $
which by Cauchy's root test means that $ P_\mu(z) $ converges absolutely.
Since this is true for all $ |z| < 1 / \sqrt{\closed{L}_2} $, we conclude that the radius of convergence
must be at least $ 1 / \sqrt{\closed{L}_2} > \sqrt{2 / (\|\mu\| \cdot (\|\mu\| + C))} $.
\end{proof}
\bigskip

For the remainder of this subsection, we will assume $ \mu $ is a probability measure.
Though, $ T $ may still be a partial order.

\begin{proof}[Proof of \cref{thm:length-pgf}]
First, we will consider when $ i = 0 $ and prove \cref{eq:init-length-pgf}.
The first run must contain at least $ X_0 $ so $ N_0 \geq 1 $.
Thus, $ \prob{N_0 = 0} = 0 = \|\mu\| - 1 = L_1 - L_0 $.
Similarly, for $ n = 1 $, we know
\begin{equation} \label{eq:start-prob}
\prob{\open{N}_1 = 1} = \prob{X_0 \nless X_1} = 1 - \prob{X_0 < X_1} = \open{L}_1 - \open{L}_2.
\end{equation}
Now, we observe that for $ n \geq 2 $,
\[ \begin{split}
\open{L}_n &= \prob{X_0 < \ldots < X_{n-1}} \\
&= \prob{X_0 < \ldots < X_{n-1} \text{ and } X_{n-1} < X_n} + \prob{X_0 < \ldots < X_{n-1} \text{ and } X_{n-1} \nless X_n} \\
&= \open{L}_{n+1} + \prob{\open{N}_0 = n}
\end{split} \]
meaning that $ \open{L}_n - \open{L}_{n+1} = \prob{\open{N}_0 = n} $.
Applying the same arguments, we get the equalities in the non-strict case as well.
Now, we can manipulate the run functions to obtain
\[ \begin{split}
1 + (Z - 1) P_\mu(Z) &= 1 + \sum_{n=0}^\infty L_n Z^{n+1} - \sum_{n=0}^\infty L_n Z^n
= 1 + \sum_{n=0}^\infty L_n Z^{n+1} - 1 - \sum_{n=0}^\infty L_{n+1} Z^{n+1} \\
&= \sum_{n=0}^\infty (L_n - L_{n+1}) Z^{n+1} = \sum_{n=0}^\infty \prob{N_0 = n} Z^{n+1} = Z \cdot G_{N_0}(Z).
\end{split} \]
This completes the proof of \cref{eq:init-length-pgf}.
\bigskip

Second, we will consider when $ i \geq 1 $ and prove \cref{eq:length-pgf}.
Since $ i \geq 1 $, we know $ \open{S}_i = \{X_{i - 1} \nless X_i\} $.
Thus,
\[ \prob{\open{S}_i} = \prob{X_{i - 1} \nless X_i} = 1 - \prob{X_{i - 1} < X_i} = 1 - \open{L}_2. \]
Now, for any $ n \geq 1 $, using symmetry and the fact that $ \open{L}_n - \open{L}_{n+1} = \prob{\open{N}_0 = 0} $, we get
\[ \begin{split}
\open{L}_n - \open{L}_{n+1} &= \prob{X_0 < \ldots < X_{n-1} \nless X_n}
= \prob{X_i < \ldots < X_{i+n-1} \nless X_{i+n}} \\
&= \prob{X_{i-1} < X_i \text{ and } X_i < \ldots < X_{i+n-1} \nless X_{i+n}} + \prob{X_{i-1} \nless X_i \text{ and } X_i < \ldots < X_{i+n-1} \nless X_{i+n}} \\
&= \open{L}_{n+1} - \open{L}_{n+2} + \prob{\open{N}_i = n \text{ and } \open{S}_i}.
\end{split} \]
Rearranging, this means that
\begin{equation} \label{eq:length-prob}
\open{L}_n - 2\open{L}_{n+1} + \open{L}_{n+2}
= \probWhen{\open{N}_i = n}{\open{S}_i} \cdot \prob{\open{S}_i}
= \left(1 - \open{L}_2\right) \probWhen{\open{N}_i = n}{\open{S}_i}
\end{equation}
and we can apply the same argument for $ \closed{N}_i $.
For $ n = 0 $, we know the length of a run is at least one so $ \prob{N_i = 0} = 0 $.
Now, recalling that $ L_1(\mu) = \|\mu\| = 1 $,
\[ \begin{split}
(Z - 1)^2 P_\mu(Z) &= \sum_{n=0}^\infty L_n Z^n - 2 \sum_{n=0}^\infty L_n Z^{n+1} + \sum_{n=0}^\infty L_n Z^{n+2} \\
&= L_0 + L_1 Z + \sum_{n=0}^\infty L_{n+2} Z^{n+2} - 2L_0 Z - \sum_{n=0}^\infty 2 L_{n+1} Z^{n+2} + \sum_{n=0}^\infty L_n Z^{n+2} \\
&= 1 - Z + \sum_{n=0}^\infty (L_{n+2} - 2 L_{n+1} + L_n) Z^{n+2} = 1 - Z + (L_2 - 2 L_1 + L_0) Z^2 + \sum_{n=1}^\infty (L_{n+2} - 2 L_{n+1} + L_n) Z^{n+2} \\
&= 1 - Z + (L_2 - 1) Z^2 + Z^2 (1 - L_2) \sum_{n=0}^\infty \probWhen{N_i = n}{S_i} Z^n \\
&= 1 - Z + (L_2 - 1) Z^2 + Z^2 (1 - L_2) \cdot G_{N_i}(Z).
\end{split} \]
Recalling \cref{eq:init-length-pgf}, we can rewrite this as
\[
G_{N_i}(Z) = \frac{Z - 1 + (1 - L_2) Z^2 + (Z - 1)^2 P_\mu(Z)}{Z^2 (1 - L_2)}
= 1 + \frac{(Z - 1)\left( \Big. 1 + (Z - 1) P_\mu(Z) \right)}{Z^2 (1 - L_2)}
= 1 + \frac{Z - 1}{Z (1 - L_2)} G_{N_0}(Z).
\]
From \cref{lem:func-conv-rad}, we know the radius of convergence of $ \open{P}_\mu(Z) $ is infinite
and so $ G_{\open{N}_i}(Z) $ also has an infinite radius of convergence.
Similarly, because the radius of convergence of $ \closed{P}_\mu(Z) $ is greater than or equal to $ \frac{1}{\|\mu\|} $,
we know the same holds for $ G_{\closed{N}_i}(Z) $.
\end{proof}

Now, we will derive the mean and variance from \cref{eq:init-length-pgf} and \cref{eq:length-pgf}.
\begin{proof}[Proof of \cref{coro:length-stats}]
First from \cref{lem:func-conv-rad}, we know that $ P_\mu(Z) $ will have a radius of convergence greater than $ \frac{1}{\|\mu\|} = 1 $
in particular $ P_\mu(1) = \sum_{n=0}^\infty L_n $ exists and we must have $ \lim_{n \to \infty} L_n = 0 $.
Thus, as $ k \to \infty $,
\[ \begin{split}
\prob{\closed{N}_i = \infty} &= \prob{X_i \leq X_{i+1} \leq \ldots} \leq \prob{X_i \leq \ldots \leq X_{i+k}} = \closed{L}_{k+1} \to 0
\quad \text{ and} \\
\prob{\open{N}_i = \infty} &= \prob{X_i < X_{i+1} < \ldots} \leq \prob{X_i < \ldots < X_{i+k}} = \open{L}_{k+1} \to 0
\end{split} \]
implying $ \prob{\closed{N}_i = \infty} = \prob{\open{N}_i = \infty} = 0 $ so $ N_i $ is almost surely finite.
\medskip

Now, to find the expected value, we can take the derivative of the PGF at one.
Using \cref{eq:init-length-pgf} and \cref{eq:length-pgf}, we know that the PGFs for $ \open{N}_i $ and $ \closed{N}_i $
can be expressed in terms of the run functions.
Because the radius of convergence of $ P_\mu(Z) $ is greater than one,
we can take infinitely many derivatives of $ P_\mu(z) $ at $ z = 1 $.
Thus,
\[ \begin{split}
\expect{N_0} = \left. \left(z \frac{d}{dz}\right) G_{N_0}(z) \Bigg. \right|_{z=1}
&= \left. \left(z \frac{d}{dz}\right) \frac{1 + (z - 1) P_\mu(z)}{z} \Bigg. \right|_{z=1} \\
&= \left. z \frac{z (P_\mu(z) + (z - 1) P'_\mu(z)) - (1 + (z - 1) P_\mu(z))}{z^2} \Bigg. \right|_{z=1} \\
&= \left. \frac{z (z - 1) P'_\mu(z) + P_\mu(z) - 1}{z} \Bigg. \right|_{z=1}
= \frac{1 \cdot (1 - 1) P'_\mu(1) + P_\mu(1) - 1}{1} = P_\mu(1) - 1
\end{split} \]
proving \cref{eq:init-length-mean}.
Similarly for the variance, we will use the PGF to find
\[ \begin{split}
\expect{N_0^2} = \left. \left(z \frac{d}{dz}\right)^2 G_{N_0}(z) \Bigg. \right|_{z=1}
&= \left. \left(z \frac{d}{dz}\right) \frac{z (z - 1) P'_\mu(z) + P_\mu(z) - 1}{z} \Bigg. \right|_{z=1} \\
&= \left. z \frac{z \left((2z - 1) P'_\mu(z) + z (z - 1) P''_\mu(z) + P'_\mu(z) \big. \right) - \left(z (z - 1) P'_\mu(z) + P_\mu(z) - 1 \big. \right)}{z^2} \Bigg. \right|_{z=1} \\
&= \left. \frac{z^2 (z - 1) P''_\mu(z) + (z + z^2) P'_\mu(z) - P_\mu(z) + 1}{z} \Bigg. \right|_{z=1} \\
&= \frac{1^2 \cdot (1 - 1) P''_\mu(1) + (1 + 1^2) P'_\mu(1) - P_\mu(1) + 1}{1} = 2 P'_\mu(1) - P_\mu(1) + 1.
\end{split} \]
Then, the variance will be
\[
\var{N_0} = \expect{N_0^2} - \expect{N_0}^2 = (2 P'_\mu(1) - P_\mu(1) + 1) - (P_\mu(1) - 1)^2
= P_\mu(1) - P_\mu(1)^2 + 2 P'_\mu(1)
\]
proving \cref{eq:init-length-var}.
\bigskip
Now, we consider $ i \geq 1 $.
Again because the radius of convergence of $ P_\mu(Z) $ is one,
we know $ P_\mu $ is infinitely differentiable at zero and $ L_2 = P_\mu''(0) / 2 $.
Using \cref{eq:length-pgf} and noting that $ G_{N_0}(1) = 1 $, we can take the first derivative to get
\[ \begin{split}
\expectWhen{N_i}{S_i} = \left. \left(z \frac{d}{dz}\right) G_{N_i}(z) \Bigg. \right|_{z=1}
&= \left. \left(z \frac{d}{dz}\right) \left[1 + \frac{z - 1}{z (1 - L_2)} G_{N_0}(z) \right] \Bigg. \right|_{z=1} \\
&= \left. z \frac{z - 1}{z (1 - L_2)} G'_{N_0}(z) + z \frac{1}{z^2 (1 - L_2)} G_{N_0}(z) \Bigg. \right|_{z=1} \\
&= \frac{1 - 1}{1 - L_2} G'_{N_0}(z) + \frac{1}{1 - L_2} G_{N_0}(1) = \frac{1}{1 - L_2} = \frac{2}{2 - P_\mu''(0)}
\end{split} \]
proving \cref{eq:length-mean}.
From \cref{eq:init-length-mean}, we get $ 1 \cdot G'_{N_0}(1) = \expect{N_0} = P_\mu(1) - 1 $ so
\[ \begin{split}
\expectWhen{N_i^2}{S_i} = \left. \left(z \frac{d}{dz}\right)^2 G_{N_i}(z) \Bigg. \right|_{z=1}
&= \frac{1}{1 - L_2} \left( \left. \left(z \frac{d}{dz}\right) \left[ (z - 1) G'_{N_0}(z) + \frac{1}{z} G_{N_0}(z) \right] \Bigg. \right|_{z=1} \right) \\
&= \frac{1}{1 - L_2} \left( \left. z (z - 1) G''_{N_0}(z) + z G'_{N_0}(z) + \frac{z}{z} G'_{N_0}(z) - \frac{z}{z^2} G_{N_0}(z) \Bigg. \right|_{z=1} \right) \\
&= \frac{1}{1 - L_2} \left( \Big. (1 - 1) G''_{N_0}(1) + G'_{N_0}(1) + G'_{N_0}(1) - G_{N_0}(1) \right) \\
&= \frac{2 G'_{N_0}(1) - 1}{1 - L_2} = \frac{2 (P_\mu(1) - 1) - 1}{1 - L_2} = \frac{2 P_\mu(1) - 3}{1 - L_2}.
\end{split} \]
Thus, the variance for $ N_i $ will be
\[ \begin{split}
\varWhen{N_i}{S_i} = \expectWhen{N_i^2}{S_i} - \expectWhen{N_i}{S_i}^2 &= \frac{2 P_\mu(1) - 3}{1 - L_2} - \frac{1}{(1 - L_2)^2} \\
&= \frac{1}{1 - L_2} \left(-3 + 2 P_\mu(1) - \frac{1}{1 - L_2} \right) \\
&= \frac{2}{2 - P_\mu''(0)} \left(-3 + 2 P_\mu(1) - \frac{2}{2 - P_\mu''(0)} \right)
\end{split} \]
proving \cref{eq:length-var}.
\end{proof}

\section{Strong Law of Large Numbers}

\begin{proof}[Proof of \cref{thm:slln}]
To prove this, we will use a generalization of the law of large numbers for weakly correlated random variables from Lyons \cite{Lyo}.
We define $ W_b $ to be the number of runs in the sequence $ \{X_i\}_{i=0}^\infty $ starting in $ [0, b) $
and $ U_{b,n} $ to be the number of runs starting in $ [0, b) $ of length $ n $.
First, notice that no run can have length zero so for any $ b $, $ U_{b,0} = 0 $ always.
Thus,
\[ \lim_{b \to \infty} \frac{U_{b,0}}{W_b} = 0 = \probWhen{N_i = 0}{S_i}. \]
Now, for each $ n \geq 1 $ and $ i \geq 0 $, we define
\[
R_i := \begin{cases} 1 & \text{if } S_i \\ 0 & \text{otherwise} \end{cases}
\qquad
Q_{i,n} := \begin{cases} 1 & \text{if } S_i \text{ and } N_i = n \\ 0 & \text{otherwise} \end{cases}.
\]
Let $ n \geq 0 $ and $ b \geq 0 $ then
\[ W_b = \sum_{i=0}^{b-1} R_i \quad \text{ and } \quad U_{b,n} = \sum_{i=0}^{b-1} Q_{i,n}. \]
Now, suppose $ j \geq i \geq 0 $ are indices with $ |i - j| > 1 $.
If $ i = 0 $ then $ R_0 = 1 $ is always true and $ \prob{R_0 = 1 \text{ and } R_j = 1} = \prob{R_j = 1} = \prob{R_0 = 1} \cdot \prob{R_j = 1} $.
Otherwise, $ i \neq 0 $ with $ j - 1 > i $ so by independence of the sequence $ \{X_i\} $
\[
\prob{\open{R}_i = 1 \text{ and } \open{R}_j = 1} = \prob{X_{i-1} \nless X_i \text{ and } X_{j-1} \nless X_j}
= \prob{X_{i-1} \nless X_i} \cdot \prob{X_{j-1} \nless X_j} = \prob{\open{R}_i = 1} \cdot \prob{\open{R}_j = 1}
\]
and similarly for $ \closed{R}_i $ and $ \closed{R}_j $.
Thus, $ R_i $ and $ R_j $ are independent if $ |i - j| > 1 $.
Similarly, suppose $ j \geq i \geq 0 $ are indices with $ |i - j| > n + 1 $.
Then, $ j - 1 > i + n $ so by independence,
\[ \begin{split}
\prob{\open{Q}_{0,n} = 1 \text{ and } \open{Q}_{j,n} = 1}
&= \prob{X_0 < \ldots < X_{n-1} \nless X_n \text{ and } X_{j-1} \nless X_j < \ldots < X_{j+n-1} \nless X_{j+n}} \\
&= \prob{X_0 < \ldots < X_{n-1} \nless X_n} \cdot \prob{X_{j-1} \nless X_j < \ldots < X_{j+n-1} \nless X_{j+n}} \\
&= \prob{\open{Q}_{0,n} = 1} \cdot \prob{\open{Q}_{j,n} = 1} \quad \text{ and} \\
\prob{\open{Q}_{i,n} = 1 \text{ and } \open{Q}_{j,n} = 1}
&= \prob{X_{i-1} \nless X_i < \ldots < X_{i+n-1} \nless X_{i+n} \text{ and } X_{j-1} \nless X_j < \ldots < X_{j+n-1} \nless X_{j+n}} \\
&= \prob{X_{i-1} \nless X_i < \ldots < X_{i+n-1} \nless X_{i+n}} \cdot \prob{X_{j-1} \nless X_j < \ldots < X_{j+n-1} \nless X_{j+n}} \\
&= \prob{\open{Q}_{i,n} = 1} \cdot \prob{\open{Q}_{j,n} = 1} \quad \text{ when } i \neq 0
\end{split} \]
and the same arguments apply for the non-strict case.
Therefore, $ Q_{i,n} $ and $ Q_{j,n} $ are independent if $ |i - j| > n + 1 $.
\bigskip

We will be using corollary 4 from Lyons \cite{Lyo}
which states that for a sequence of random variables $ \{Y_i\}_{i=1}^\infty $ with $ |Y_i| \leq 1 $ a.s.,
\[ \lim_{n \to \infty} \frac{1}{n} \sum_{i=1}^n Y_i = 0 \quad \text{ a.s.} \]
if there exists some non-negative function $ \Phi(k) $ on the integers such that $ \expect{Y_i Y_j} \leq \Phi(|i - j|) $ and $ \sum_{k=1}^\infty \Phi(k) / k < \infty $.
\medskip

For each $ i \geq 0 $ and $ n \geq 1 $, we define $ A_i := R_i - \expect{R_i} $ and $ B_{i,n} := Q_{i,n} - \expect{Q_{i,n}} $
so that $ \expect{A_i} = \expect{B_{i,n}} = 0 $.
Since $ R_i, Q_{i,n} \in [0, 1] $ surely, we know $ |A_i|, |B_{i,n}| \leq 1 $.
Observe that for any indices $ i $ and $ j  $,
\[
\expect{A_i A_j} \leq \expect{|A_i| \cdot |A_j|} \leq 1 \quad \text{ when } |i - j| \leq 1
\quad \text{ and } \quad
\expect{B_{i,n} B_{j,n}} \leq \expect{|B_{i,n}| \cdot |B_{j,n}|} \leq 1 \quad \text{ when } |i - j| \leq n + 1
\]
\[
\expect{A_i A_j} = \expect{A_i} \expect{A_j} = 0 \quad \text{ when } |i - j| > 1
\quad \text{ and } \quad
\expect{B_{i,n} B_{j,n}} = \expect{B_{i,n}} \expect{B_{j,n}} = 0 \quad \text{ when } |i - j| > n + 1.
\]
We then define
\[
\Phi_A(k) := \begin{cases} 1 & k \leq 1 \\ 0 & k > 1 \end{cases}
\quad \text{ and } \quad
\Phi_{B,n}(k) := \begin{cases} 1 & k \leq n + 1 \\ 0 & k > n + 1 \end{cases}
\]
which fulfill the requirements of Lyons' corollary 4.
Therefore, the strong law of large numbers holds for $ A_i $ and $ B_{i,n} $ so almost surely
\[
\lim_{b \to \infty} \frac{W_b - \expect{W_b}}{b} = \lim_{b \to \infty} \frac{1}{b} \sum_{i=0}^{b-1} A_i = 0 \quad \text{ and } \quad
\lim_{b \to \infty} \frac{U_{b,n} - \expect{U_{b,n}}}{b} = \lim_{b \to \infty} \frac{1}{b} \sum_{i=0}^{b-1} B_{i,n} = 0.
\]
From \cref{eq:start-prob} and \cref{eq:length-prob},
we know $ \prob{R_i = 1} = 1 - L_2 $ and $ \prob{R_i = 1 \text{ and } N_i = n} = L_{n+2} - 2 L_{n+1} + L_n $ for all $ i \geq 1 $ so
\[ \begin{split}
\expect{\frac{W_b}{b}} = \frac{1}{b} \sum_{i=0}^{b-1} \expect{R_i}
&= \frac{\prob{R_0 = 1}}{b} + \frac{1}{b} \sum_{i=1}^{b-1} \prob{R_i = 1} = \frac{1}{b} + \frac{b - 1}{b} (1 - L_2)
\quad \text{ and} \\
\expect{\frac{U_{b,n}}{b}} = \frac{1}{b} \sum_{i=0}^{b-1} \expect{Q_{i,n}}
&= \frac{\prob{N_0 = n}}{b} + \frac{1}{b} \sum_{i=1}^{b-1} \prob{R_i = 1 \text{ and } N_i = n} \\
&= \frac{\prob{N_0 = n}}{b} + \frac{b - 1}{b} (L_{n+2} - 2L_{n+1} + L_n).
\end{split} \]
Using these expected values in the limits, we know almost surely that
\[ \begin{split}
\lim_{b \to \infty} \frac{W_b}{b} &= \lim_{b \to \infty} \left( \frac{1}{b} + \frac{b-1}{b} (1 - L_2) \right) = 1 - L_2
\qquad \text{ and} \\
\lim_{b \to \infty} \frac{U_{b,n}}{b} &= \lim_{b \to \infty} \left( \frac{\prob{N_0 = n}}{b} + \frac{b-1}{b} (L_{n+2} - 2L_{n+1} + L_n) \right) = L_{n+2} - 2L_{n+1} + L_n
\end{split} \]
which implies that their quotient is almost surely
\[
\lim_{b \to \infty} \frac{U_{b,n}}{W_b} = \lim_{b \to \infty} \frac{U_{b,n} / b}{W_b / b}
= \frac{L_{n+2} - 2L_{n+1} + L_n}{1 - L_2} = \probWhen{N_i = n}{S_i}
\quad \text{ for any } i \geq 1.
\]
\end{proof}

\section{Run Functions of Limits of Measures}
In this section, we will prove that under certain circumstances
the convergence of a sequence of measures
implies the uniform convergence of their run functions
on bounded intervals.

\begin{lem} \label{lem:func-diff-bound}
Suppose $ \mu_1 $ and $ \mu_2 $ are finite positive measures on $ T $.
For all $ z \in \mathbb{C} $, if $ P_{\mu_1}(|z|) $ and $ P_{\mu_2}(|z|) $ exist then
\[ \abs{P_{\mu_1}(z) - P_{\mu_2}(z)} \leq \|\mu_1 - \mu_2\| \cdot |z| \cdot P_{\mu_1}(|z|) \cdot P_{\mu_2}(|z|). \]
\end{lem}
\begin{proof}
First, we will bound the distance of $ L_n(\mu_1) $ and $ L_n(\mu_2) $.
Let $ n \geq 1 $ and observe that the following sum telescopes
\[ \begin{split}
\sum_{k=0}^{n-1} \mu_1^k \times (\mu_1 - \mu_2) \times \mu_2^{n-k-1}
&= \left(\sum_{k=0}^{n-1} \mu_1^{k+1} \times \mu_2^{n-k-1} \right) - \left(\sum_{k=0}^{n-1} \mu_1^k \times \mu_2^{n-k} \right) \\
&= \left(\sum_{k=1}^n \mu_1^k \times \mu_2^{n-k} \right) - \left(\sum_{k=0}^{n-1} \mu_1^k \times \mu_2^{n-k} \right)
= \mu_1^n - \mu_2^n
\end{split} \]
Thus, we can bound $ \abs{\open{L}_n(\mu_1) - \open{L}_n(\mu_2)} $ as
\begin{equation} \label{eq:coef-diff-bound} \begin{split}
\abs{\open{L}_n(\mu_1) - \open{L}_n(\mu_2)}
&\leq \abs{\mu_1^n - \mu_2^n}(\{a_1 < \ldots < a_n\}_{T^n}) \\
&\leq \sum_{k=0}^{n-1} \abs{\mu_1^k \times (\mu_1 - \mu_2) \times \mu_2^{n-k-1}}(\{a_1 < \ldots < a_n\}_{T^n}) \\
&\leq \sum_{k=0}^{n-1} \abs{\mu_1^k \times (\mu_1 - \mu_2) \times \mu_2^{n-k-1}}(\{a_1 < \ldots < a_k, a_{k+2} < \ldots < a_n\}_{T^n}) \\
&= \sum_{k=0}^{n-1} \abs{\mu_1^k}(\{a_1 < \ldots < a_k\}_{T^k}) \cdot \abs{\mu_1 - \mu_2}(T) \cdot \abs{\mu_2^{n-k-1}}(\{a_{k+2} < \ldots < a_n\}_{T^{n-k-1}}) \\
&= \sum_{k=0}^{n-1} \open{L}_k(\mu_1) \cdot \|\mu_1 - \mu_2\| \cdot \open{L}_{n-k-1}(\mu_1)
\end{split} \end{equation}
and the same argument holds for $ \abs{\closed{L}_n(\mu_1) - \closed{L}_n(\mu_2)} $.
\bigskip

Now, let $ z \in \mathbb{C} $ such that $ P_{\mu_1}(|z|) $ and $ P_{\mu_2}(|z|) $ converge.
Note that $ L_0(\mu_1) = 1 = L_0(\mu_2) $ so they will cancel in the difference.
Then, using \cref{eq:coef-diff-bound}
\[ \begin{split}
\abs{P_{\mu_1}(z) - P_{\mu_2}(z)} &= \abs{\sum_{n=1}^\infty (L_n(\mu_1) - L_n(\mu_2)) z^n} \\
&\leq \sum_{n=1}^\infty \abs{L_n(\mu_1) - L_n(\mu_2)} \cdot |z|^n
= |z| \sum_{n=1}^\infty |z|^{n-1} \cdot \|\mu_1 - \mu_2\| \sum_{j=0}^{n-1} L_j(\mu_1) \cdot L_{n-j-1}(\mu_2) \\
&= \|\mu_1 - \mu_2\| \cdot |z| \sum_{n=0}^\infty \sum_{j=0}^n (L_j(\mu_1) \cdot L_{n-j}(\mu_2)) |z|^n.
\end{split} \]
Using the substitution $ j = n - k $, we can say
\[ \begin{split}
\sum_{n=0}^\infty \sum_{j=0}^n (L_j(\mu_1) \cdot L_{n-j}(\mu_2)) |z|^n
&= \sum_{j=0}^\infty \sum_{k=0}^\infty (L_j(\mu_1) \cdot L_k(\mu_2)) |z|^{j + k} \\
&= \left(\sum_{j=0}^\infty L_j(\mu_1) |z|^j \right)\left(\sum_{k=0}^\infty L_k(\mu_2) |z|^k \right)
= P_{\mu_1}(|z|) \cdot P_{\mu_2}(|z|)
\end{split} \]
proving the desired bound.
\end{proof}

\begin{prop} \label{prop:func-limit}
Suppose that $ \mu, \mu_1, \mu_2, \ldots $ are finite non-negative measures on $ T $
with $ \mu_i \leq \mu_j $ for $ i \leq j $ and $ \mu_k \to \mu $ in the total variation norm.
If $ R $ is the radius of convergence of $ P_\mu(Z) $
then for all $ z \in \mathbb{C} $ with $ |z| < R $, $ P_{\mu_k}(|z|) $ exists for all $ k $.
Additionally, for any $ 0 < b < R $, $ P_{\mu_k}(z) \to P_\mu(z) $ uniformly on $ \{|z| \leq b\} $.
\end{prop}
\begin{proof}
First, we will show that $ P_{\mu_k}(|z|) $ exists when $ |z| < R $.
For the sctrict case, by \cref{lem:func-conv-rad}, the radii of convergence of $ \open{P}_\mu(Z) $
and $ \open{P}_{\mu_1}(Z), \open{P}_{\mu_2}(Z), \ldots $ are all infinite
so $ \open{P}_{\mu_k}(|z|) $ always exists.
For any $ n $, by monotonicity, we know
\[ \closed{L}_n(\mu_k) = \mu_k^n(\{a_1 \leq \ldots \leq a_n\}_{T^n}) \leq \mu^n(\{a_1 \leq \ldots \leq a_n\}_{T^n}) = \closed{L}_n(\mu). \]
Thus, if $ z \in \mathbb{C} $ with $ |z| < R $ then $ \closed{P}_{\mu_k}(|z|) \leq \closed{P}_\mu(|z|) < \infty $.
\bigskip

Next, we show uniform convergence on bounded sets.
Let $ 0 < b < R $ then $ P_\mu(b) < \infty $.
Now, for all $ z \in \mathbb{C} $ with $ |z| \leq b $, by monotonicity $ P_{\mu_k}(|z|) \leq P_\mu(|z|) \leq P_\mu(b) $.
Then, by \cref{lem:func-diff-bound}
\[
\abs{P_{\mu_k}(z) - P_\mu(z)} \leq \|\mu_k - \mu\| \cdot |z| \cdot P_{\mu_k}(|z|) \cdot P_\mu(|z|)
\leq \|\mu_k - \mu\| \cdot b \cdot P_\mu^2(b) \to 0.
\]
Since the bound is independent of $ z $, this implies that $ \open{P}_{\mu_k}(z) \to \open{P}_\mu(z) $ uniformly on $ \{|z| < b\} $.
\end{proof}

\section{Combining Run Functions}

In this section, we will show how the run functions
behave under series and parallel composition.
We begin by proving the result for finite $ \Gamma $.
It is sufficient to consider when $ \Gamma = \{1, 2\} $
because $ \otimes $ and $ \oplus $ are associative and we can extend by induction.
\medskip

\begin{lem} \label{lem:fin-series-func}
If $ \mu_1 $ and $ \mu_2 $ are finite measures over partial orders $ T_1 $ and $ T_2 $, respectively,
then $ P_{\mu_1 \otimes \mu_2}(Z) = P_{\mu_1}(Z) \cdot P_{\mu_2}(Z) $ as formal power series.
\end{lem}
\begin{proof}
For all integers $ n \geq 1 $ and $ 0 \leq k \leq n $, we define $ A_n, B_n \subseteq (T_1 \otimes T_2)^n $ as
\[ \begin{split}
A_{k,n} &:= \begin{cases}
\hfil T_2 & \text{ for } n = 1 \text{ and } k = 0 \\
\hfil T_1 & \text{ for } n = 1 \text{ and } k = 1 \\
\{\langle a_1, \ldots, a_n \rangle \in T_1^k \times T_2^{n-k} \;:\; a_1 < \ldots < a_n\} & \text{ otherwise}
\end{cases} \\
B_n &:= \begin{cases}
\hfil T_1 \otimes T_2 & \text{ for } n = 1 \\
\{\langle a_1, \ldots, a_n \rangle \in (T_1 \otimes T_2)^n \;:\; a_1 < \ldots < a_n\} & \text{ otherwise}
\end{cases}.
\end{split} \]
Here, we are taking $ T_1^0 \times T_2^n $ and $ T_1^n \times T_2^0 $ to mean $ T_2^n $ and $ T_1^n $, respectively.
Next, we define $ \tilde{B}_n \subseteq (T_1 \otimes T_2)^n $ as
\[ \tilde{B}_n := \bigcup_{k=0}^n A_{k,n}. \]
We will now prove that $ B_n = \tilde{B}_n $.
When $ n = 1 $, we have $ B_0 = T_1 \otimes T_2 = T_1 \cup T_2 = \tilde{B}_0 $.
When $ n \geq 2 $, let $ a = \langle a_1, \ldots, a_n \rangle \in B_n $ so that $ a_1 < \ldots < a_n $.
We define
\[ k := \max\{i \in \{1, \ldots, n\} \;:\; c_i \in T_1\} \]
or $ k = 0 $ if $ a_1, \ldots, a_n \in T_2 $.
For all $ i \leq k $, $ a_i \leq a_k \in T_1 $ so $ a_i \in T_1 $.
Also for all $ i > k $, by maximality, $ a_i \notin T_1 $ so $ a_i \in T_2 $.
Thus, $ a_1, \ldots, a_k \in T_1 $ and $ a_{k+1}, \ldots, a_n \in T_2 $
implying that $ a \in T_1^k \times T_2^{n-k} $.
Because $ a_1 < \ldots < a_n $, this means $ a \in A_{k,n} \subseteq \tilde{B}_n $ and so $ B_n \subseteq \tilde{B}_n $.
Conversely, let $ a = \langle a_1, \ldots, a_n \rangle \in \tilde{B}_n $ with $ a \in A_{k,n} $.
Then, $ a_1 < \ldots < a_n $ meaning $ a \in B_n $ and so $ \tilde{B}_n \subseteq B_n $.
Hence, $ B_n = \tilde{B}_n $.
\bigskip

Next, we show that the collection $ \{A_{k,n}\}_{k=0}^n $ is disjoint.
Let $ j, k \in \{0, \ldots, n\} $ with $ j < k $.
Assume for sake of contradiction that $ a \in A_{j,n} \cap A_{k,n} $.
Since $ a \in A_{k,n} $ with $ k > 0 $, $ a_k \in T_1 $.
However since $ a \in A_{j,n} $, we know $ a_{j+1}, \ldots, a_k, \ldots, a_n \in T_2 $ which is a contradiction.
We conclude that $ A_{j,n} \cap A_{k,n} = \emptyset $.
Thus, $ \open{L}_0(\mu_1 \otimes \mu_2) = 1 = \open{L}_0(\mu_1) \cdot \open{L}_0(\mu_2) $ and for $ n \geq 1 $,
\[ \begin{split}
\open{L}_n(\mu_1 \otimes \mu_2) &= (\mu_1 \otimes \mu_2)^n(B_n)
= (\mu_1 \otimes \mu_2)^n\left(\bigcup_{k=0}^n A_{k,n} \right)
= \sum_{k=0}^n (\mu_1 \otimes \mu_2)^n(A_{k,n}) \\
&= \sum_{k=0}^n (\mu_1 \otimes \mu_2)^k\left(\{a_1 < \ldots < a_k\}_{T_1^k} \right)
\cdot (\mu_1 \otimes \mu_2)^{n-k}\left(\{a_{k+1} < \ldots < a_n\}_{T_2^{n-k}} \right) \\
&= \sum_{k=0}^n \mu_1^k\left(\{a_1 < \ldots < a_k\}_{T_1^k} \right) \cdot \mu_2^{n-k}\left(\{a_{k+1} < \ldots < a_n\}_{T_2^{n-k}} \right)
= \sum_{k=0}^n \open{L}_k(\mu_1) \cdot \open{L}_{n-k}(\mu_2).
\end{split} \]
We can apply the same argument to $ \closed{L}_n(\mu_1 \otimes \mu_2) $ as well.
Therefore, using the substitution $ n = j + k $, the run functions become
\[ \begin{split}
P_{\mu_1 \otimes \mu_2}(Z) &= \sum_{n=0}^\infty L_n(\mu_1 \otimes \mu_2) Z^n
= \sum_{n=0}^\infty \sum_{j=0}^n (L_j(\mu_1) \cdot L_{n-j}(\mu_2)) Z^n \\
&= \sum_{j=0}^\infty \sum_{k=0}^\infty (L_j(\mu_1) \cdot L_k(\mu_2)) Z^{j + k}
= \left(\sum_{j=0}^\infty L_j(\mu_1) Z^j\right) \left(\sum_{k=0}^\infty L_k(\mu_2) Z^k\right)
= P_{\mu_1}(Z) \cdot P_{\mu_2}(Z).
\end{split} \]
\end{proof}

\begin{lem} \label{lem:fin-para-func}
If $ \mu_1 $ and $ \mu_2 $ are finite measures over partial orders $ T_1 $ and $ T_2 $, respectively,
then $ P_{\mu_1 \oplus \mu_2}(Z) - 1 = (P_{\mu_1}(Z) - 1) + (P_{\mu_2}(Z) - 1) $ as formal power series.
\end{lem}
\begin{proof}
For all integers $ n \geq 1 $, we define $ A_n, B_n, C_n \subseteq T $ as
\[
A_n := \{a_1 < \ldots < a_n\}_{T_1^n}
\quad \text{ and } \quad
B_n := \{a_1 < \ldots < a_n\}_{T_2^n}
\quad \text{ and } \quad
C_n := \{a_1 < \ldots < a_n\}_{(T_1 \oplus T_2)^n}
\]
Immediately, we have $ A_n, B_n \subseteq C_n $.
Let $ a = \langle a_1, \ldots, a_n \rangle \in C_n $ then
either $ a_1 \in T_1 $ or $ a_1 \in T_2 $.
If $ a_1 \in T_1 $ then since $ a_2, \ldots, a_n $ are comparable to $ a_1 $,
we know $ a_2, \ldots, a_n \in T_1 $ and so $ a \in A_n $.
Otherwise, if $ a_1 \in T_2 $ then $ a \in B_n $.
Hence, we have $ C_n \subseteq A_n \cup B_n $ and so $ C_n = A_n \cup B_n $.
\medskip

Notice that $ A_n \cap B_n \subseteq T_1^n \cap T_2^n = \emptyset $ thus
\[
\open{L}_n(\mu_1 \oplus \mu_2) = (\mu_1 \oplus \mu_2)^n(C_n)
= (\mu_1 \oplus \mu_2)^n(A_n) + (\mu_1 \oplus \mu_2)^n(B_n)
= \mu_1^n(A_n) + \mu_2^n(B_n) = \open{L}_n(\mu_1) + \open{L}_n(\mu_2).
\]
We may apply this same argument to the non-strict case
so $ L_n(\mu_1 \oplus \mu_2) = L_n(\mu_1) + L_n(\mu_2) $ for $ n \geq 1 $.
Therefore, the run functions will be
\[
P_{\mu_1 \oplus \mu_2}(Z) - 1 = \sum_{n=1}^\infty L_n(\mu_1 \oplus \mu_2) Z^n
= \sum_{n=1}^\infty L_n(\mu_1) Z^n + \sum_{n=1}^\infty L_n(\mu_2) Z^n
= (P_{\mu_1}(Z) - 1) + (P_{\mu_2}(Z) - 1).
\]
\end{proof}

\begin{prop} \label{prop:comb-funcs}
Suppose $ M = \{\mu_\gamma\}_{\gamma \in \Gamma} $ is a collection of finite non-negative measures on
the collection of partial orders $ \mathcal{T} = \{T_\gamma\}_{\gamma \in \Gamma} $.
We assume that the sum of the norms of $ \{\mu_\gamma\} $ is finite.
Let $ R_\otimes $ and $ R_\oplus $ be the radii of convergence of $ P_{\bigotimes M} $ and $ P_{\bigoplus M} $, respectively.
If $ \Gamma $ is countable then for all $ z \in \mathbb{C} $,
\[
P_{\bigotimes M}(z) = \prod_{\gamma \in \Gamma} P_{\mu_\gamma}(z)
\; \text{ when } |z| < R_\otimes
\quad \text{ and } \quad
P_{\bigoplus M}(z) - 1 = \sum_{\gamma \in \Gamma} (P_{\mu_\gamma}(z) - 1)
\; \text{ when } |z| < R_\oplus.
\]
\end{prop}
\begin{proof}
Since $ \Gamma $ is countable, let $ \{\gamma_1, \gamma_2, \ldots\} $ be an enumeration for $ \Gamma $.
We will say $ \mu^\otimes = \bigotimes M $ and $ \mu^\oplus = \bigoplus M $.
For every integer $ k \geq 1 $, we define a new collection of measures $ \tilde{M}_k := \{\tilde{\mu}_{k,\gamma}\}_{\gamma \in \Gamma} $ where
\[
\tilde{\mu}_{k,\gamma} = \begin{cases}
\mu_\gamma & \text{ when } \gamma \in \{\gamma_1, \ldots, \gamma_k\} \\
0 & \text{ otherwise}
\end{cases}.
\]
Then, we can define $ \mu^\otimes_k := \bigotimes \tilde{M}_k $ and $ \mu^\oplus_k := \bigoplus \tilde{M}_k $
as finite non-negative measures on $ \bigotimes \mathcal{T} $ and $ \bigoplus \mathcal{T} $, respectively.
Using \cref{lem:fin-series-func} and \cref{lem:fin-para-func}, we know
\[
P_{\mu^\otimes_k}(Z) = \prod_{j=1}^k P_{\mu_{\gamma_j}}(Z) \quad \text{ and } \quad
P_{\mu^\oplus_k}(Z) - 1 = \sum_{j=1}^k (P_{\mu_{\gamma_j}}(Z) - 1).
\]
Now, notice that for $ j \leq k $, $ \mu^\otimes_j \leq \mu^\otimes_k $ and $ \mu^\oplus_j \leq \mu^\oplus_k $ as well as
when $ k \to \infty $,
\[ \norm{\mu^\otimes_k - \mu_k}, \norm{\mu^\oplus_k - \mu_k} \leq \norm{\sum_{j=k+1}^\infty \mu_{\gamma_j}} \to 0. \]
since $ \sum_\gamma \|\mu_\gamma\| < \infty $.
Thus by \cref{prop:func-limit}, for all $ z \in \mathbb{C} $,
\[ \begin{split}
P_{\mu^\otimes}(z) &= \lim_{k \to \infty} P_{\mu^\otimes_k}(z)
= \lim_{k \to \infty} \prod_{j=1}^k P_{\mu_{\gamma_j}}(z)
= \prod_{\gamma \in \Gamma} P_{\mu_\gamma}(z)
\quad \text{ when } |z| < R^\otimes \text{ and} \\
P_{\mu^\oplus}(z) - 1 &= \lim_{k \to \infty} \left[ P_{\mu^\oplus_k}(z) - 1 \right]
= \lim_{k \to \infty} \sum_{j=1}^k \left(P_{\mu_{\gamma_j}}(z) - 1 \right)
= \sum_{\gamma \in \Gamma} \left(P_{\mu_\gamma}(z) - 1 \right)
\quad \text{ when } |z| < R^\oplus.
\end{split} \]
\end{proof}

\section{Run Lengths in Total Orders and Countably Series-Parallel Partial Orders}

First, we will consider the case of a total order with no atoms
and then generalize to the atomic case.

\begin{lem} \label{lem:total-diffuse-func}
Suppose $ T $ is a total order and $ \mu $ is a finite non-negative measure on $ T $.
If $ \mu $ has no atoms then $ \open{P}_\mu(Z) = \closed{P}_\mu(Z) = e^{\|\mu\| Z} $.
\end{lem}
\begin{proof}
Let $ \Sigma_n $ be the set of permutations on $ \{1, \ldots, n\} $.
Then, for every $ \sigma \in \Sigma_n $, we define the subset
\[ A_\sigma := \{a_{\sigma(1)} < \ldots < a_{\sigma(n)}\}_{T^n} \]
so that $ \open{L}_n = \mu^n(A_\sigma) $.
Notice that if $ a_{\sigma(1)} < \ldots < a_{\sigma(n)} $ then no other permutation of $ a_1, \ldots, a_n $
will be strictly ascending so $ A_\sigma $ is disjoint from all other $ A_{\sigma'} $.
Next, we define
\[ D := \{a_1, \ldots, a_n \text{ are pairwise distinct}\}_{T^n}. \]
Because $ T $ is a total order, for any tuple $ a = \langle a_1, \ldots, a_n \rangle \in D $, there is a permutation $ \sigma $
such that $ a_{\sigma(1)} < \ldots < a_{\sigma(n)} $ and therefore $ a \in A_\sigma $.
Conversely, for any $ a \in A_\sigma $, $ a_{\sigma(1)} < \ldots < a_{\sigma(n)} $ so $ a_1, \ldots, a_n $ are distinct and $ a \in D $.
Thus, $ D = \bigcup_\sigma A_\sigma $ and by disjointness
\[
\mu^n(D) = \sum_{\sigma \in \Sigma_n} \mu^n(A_\sigma)
= \sum_{\sigma \in \Sigma_n} \open{L}_n
= |\Sigma_n| \cdot \open{L}_n = n! \cdot \open{L}_n.
\]
Now, for every $ i, j \in \{1, \ldots, n\} $, we define $ E_{i,j} := \{a_i = a_j\}_{T^n} $.
By symmetry, we know $ \mu^n(E_{i,j}) = \mu^n(E_{1,2}) $ and by diffuseness (lack of atoms)
\[
\mu^n(E_{1,2}) = \mu^n(\{a_1 = a_2\}_{T^2} \times T^{n-2}) = \mu^2(\{a_1 = a_2\}_{T^2}) \cdot \mu^{n-2}(T^{n-2})
= \|\mu\|^{n-2} \cdot \int \mu(\{t\}) \dr\mu(t)= 0.
\]
Next, for every $ a = \langle a_1, \ldots, a_n \rangle \in D^c $,
there exist $ i, j \in n $ with $ a_i = a_j $ so $ a \in E_{i,j} $.
Thus,
\[ D^c \subseteq \bigcup_{i,j \in [1,n]} E_{i,j} \qquad \text{ implying } \qquad \mu^n(D^c) \leq \mu^n\left(\bigcup_{i,j} E_{i,j}\right) = 0 \]
and so $ \|\mu\|^n = \mu^n(T^n) = \mu^n(D) + \mu^n(D^c) = n! \cdot \open{L}_n + 0 $
showing that $ \open{L}_n = \frac{\|\mu\|^n}{n!} $.
\medskip

Now, we address the case of $ \closed{L}_n $.
We define $ A = \{a_1 < \ldots < a_n\} $ and $ B = \{a_1 \leq \ldots \leq a_n\} $.
From above, we know $ \mu^n(A) = \frac{\|\mu\|^n}{n!} $.
If $ a = \langle a_1, \ldots, a_n \rangle \in B \setminus A $ then
there must exist $ i \in \{1, \ldots, n\} $ with $ a_i = a_{i+1} $ so $ a \in E_{i,i+1} $.
Thus, $ B \setminus A \subseteq \bigcup_{i,j} E_{i,j} $ and $ \mu^n(B \setminus A) = 0 $.
Therefore, we have
\[ \closed{L}_n = \mu^n(B) = \mu^n(A) + \mu^n(B \setminus A) = \frac{\|\mu\|^n}{n!} + 0. \]
Finally, the run functions will be
\[ \closed{P}_\mu(Z) = \open{P}_\mu(Z) = \sum_{n=0}^\infty \frac{\|\mu\|^n}{n!} Z^n = e^{\|\mu\| Z}. \]
\end{proof}
\bigskip

As before, we assume $ \mathcal{A} $ is the set of atoms of $ \mu $
and that for every $ \alpha \in \mathcal{A} $, the mass of $ \alpha $ is $ m_\alpha = \mu(\{\alpha\}) $.
Note that because $ \sum_{\alpha \in \mathcal{A}} m_\alpha \leq \|\mu\| < \infty $, we must have at most countably many atoms.

\begin{proof}[Proof of \cref{prop:total-run-funcs}]
First, we will prove the result when $ \mathcal{A} $ is finite.
We enumerate the atoms as $ \{\alpha_1, \ldots, \alpha_k\} = \mathcal{A} $.
Because $ T $ is total, we can define $ R_0, \ldots, R_k \subseteq T $ to separate $ T $ into
\[
R_0 \cup \{\alpha_1\} \cup R_1 \cup \{\alpha_2\} \cup \ldots \cup \{\alpha_k\} \cup R_k :=
(-\infty, \alpha_1) \cup \{\alpha_1\} \cup (\alpha_1, \alpha_2) \cup \{\alpha_2\} \cup \ldots \cup \{\alpha_k\} \cup (\alpha_k, \infty) = T.
\]
Treating each section as its own total order, we have $ T = R_0 \otimes \{\alpha_1\} \otimes \ldots \otimes \{\alpha_k\} \otimes R_k $.
Next, we take $ \lambda_0, \ldots, \lambda_k $ to be the restrictions of $ \mu $ to $ R_0, \ldots, R_k $
and $ \delta_1, \ldots, \delta_k $ to be the restrictions of $ \mu $ to $ \{\alpha_1\}, \ldots, \{\alpha_k\} $.
Then, $ \mu = \lambda_0 \otimes \delta_1 \otimes \ldots \otimes \delta_k \otimes \lambda_k $ so by \cref{lem:fin-series-func},
\[ P_\mu(Z) = P_{\lambda_0}(Z) \cdot P_{\delta_1}(Z) \cdots P_{\delta_k}(Z) \cdot P_{\lambda_k}(Z). \]
Because $ R_0 \cup \ldots \cup R_k \subseteq T \setminus \mathcal{A} $,
we know $ R_0 \cup \ldots \cup R_k $ will contain no atoms of $ \mu $.
Thus, every one of $ \lambda_0, \ldots, \lambda_k $ must be diffuse and so \cref{lem:total-diffuse-func}
tells us that $ P_{\lambda_i}(Z) = e^{\|\lambda_i\| Z} = e^{\mu(R_i) Z} $ for all $ 0 \leq i \leq k $.
\bigskip

Next, let $ 1 \leq i \leq k $ and we will consider $ \delta_i $.
For the strict case,
we know $ \open{L}_0(\delta_i) = 1 $ and $ \open{L}_1(\delta_i) = \|\delta_i\| = \mu(\{\alpha_i\}) = m_{\alpha_i} $.
Then, for all $ n \geq 2 $, because $ \{\alpha_i\} $ does not have $ n $ distinct elements,
\[ \open{L}_n(\delta_i) = \delta_i^n(\{a_1 < \ldots < a_n\}_{\{\alpha_i\}^n}) = 0. \]
Thus, $ \open{P}_{\delta_i}(Z) = 1 + m_{\alpha_i} Z $.
Next, for the non-strict case,
we have $ \closed{L}_0(\delta_i) = 1 $ and for all $ n \geq 1 $, all elements of $ \{\alpha_i\} $ are equal so
\[
\closed{L}_n(\delta_i) = \delta_i^n(\{a_1 \leq \ldots \leq a_n\}_{\{\alpha_i\}^n})
= \delta_i^n(\{\alpha_i\}^n) = \|\delta_i\|^n = m_{\alpha_i}^n.
\]
Thus, $ \closed{P}_{\delta_i}(Z) = \sum_{n=0}^\infty m_{\alpha_i}^n Z^n = \frac{1}{1 - m_{\alpha_i} Z} $.
Finally, using these results in the above product, we get
\[ \begin{split}
\closed{P}_\mu(Z) &= e^{\mu(S_0) Z} \cdot \frac{1}{1 - m_{\alpha_1} Z} \cdots \frac{1}{1 - m_{\alpha_k} Z} \cdot e^{\mu(S_k) Z}
= e^{\mu(S_0 \cup \ldots \cup S_k) Z} \cdot \prod_{i=1}^k \frac{1}{1 - m_{\alpha_i} Z}
= e^{m_d Z} \prod_{\alpha \in \mathcal{A}} \frac{1}{1 - m_\alpha Z}
\quad \text{ and} \\
\open{P}_\mu(Z) &= e^{\mu(S_0) Z} \cdot (1 + m_{\alpha_1} Z) \cdots (1 + m_{\alpha_k} Z) \cdot e^{\mu(S_k) Z}
= e^{\mu(S_0 \cup \ldots \cup S_k) Z} \cdot \prod_{i=1}^k (1 + m_{\alpha_i} Z)
= e^{m_d Z} \prod_{\alpha \in \mathcal{A}} (1 + m_\alpha Z).
\end{split} \]
which proves the result in the finite case.
\bigskip

Second, we will extend our result to countable atoms using a limit.
We enumerate the atoms as $ \{\alpha_1, \alpha_2, \ldots\} = \mathcal{A} $.
For each $ k \geq 1 $, we define the measurable set $ F_k = T \setminus \{\alpha_k, \alpha_{k+1}, \ldots\} $.
Next, we define $ \mu_k = \mu \big|_{F_k} $ so that $ \mu_k $ is the restriction of $ \mu $ to $ F_k $.
Then, $ \mu_k $ will have exactly the atoms $ \{\alpha_1, \ldots, \alpha_{k-1}\} $.
So the diffuse mass of $ \mu_k $ will be
\[
\|\mu_k\| - \sum_{i=1}^{k-1} m_{\alpha_i} = \|\mu\| - \sum_{i=k}^\infty m_{\alpha_i} - \sum_{i=1}^{k-1} m_{\alpha_i}
= \|\mu\| - \sum_{i=1}^\infty m_{\alpha_i} = m_d
\]
where $ m_d $ is the diffuse mass of $ \mu $.
Using the results for the finite case, we have
\[
\open{P}_{\mu_k}(Z) = e^{m_d Z} \prod_{i=1}^{k-1} (1 + m_{\alpha_i} Z)
\quad \text{ and } \quad
\closed{P}_{\mu_k}(Z) = e^{m_d Z} \prod_{i=1}^{k-1} \frac{1}{1 - m_{\alpha_i} Z}.
\]
Next, we observe that as $ k \to \infty $
\[
\|\mu - \mu_k\| = \left\|\mu \big|_{F_k^c} + \mu \big|_{F_k} - \mu \big|_{F_k} \right\|
= \left\| \mu \big|_{F_k^c} \right\| = \mu(F_k^c) = \sum_{i=k}^\infty m_{\alpha_i} \to 0
\]
so $ \mu_k \to \mu $.
Because $ \{F_k\}_{k=1}^\infty $ is an increasing sequence of sets,
we know $ \{\mu_k\}_{k=1}^\infty $ is an increasing sequence of measures.
From \cref{lem:func-conv-rad}, we know that the radius of convergence of $ \open{P}_\mu(Z) $ is infinite
and the radius of convergence of $ \closed{P}_\mu(Z) $ is greater than $ \frac{1}{\|\mu\|} $.
Thus by \cref{prop:func-limit}, for all $ z \in \mathbb{C} $,
\[ \begin{split}
\open{P}_\mu(z) &= \lim_{k \to \infty} \open{P}_{\mu_k}(z)
= \lim_{k \to \infty} e^{m_d z} \prod_{i=1}^{k-1} (1 + m_{\alpha_i} z)
= e^{m_d z} \prod_{\alpha \in \mathcal{A}} (1 + m_\alpha z)
\quad \text{ and} \\
\text{if } |z| \leq \frac{1}{\|\mu\|} \text{ then } \quad
\closed{P}_\mu(z) &= \lim_{k \to \infty} \closed{P}_{\mu_k}(z)
= \lim_{k \to \infty} e^{m_d z} \prod_{i=1}^{k-1} \frac{1}{1 - m_{\alpha_i} z}
= e^{m_d z} \prod_{\alpha \in \mathcal{A}} \frac{1}{1 - m_\alpha z}.
\end{split} \]
\end{proof}

\begin{lem} \label{lem:total-coef2}
For any finite non-negative measure $ \mu $ on a total order,
\begin{equation} \label{eq:coef2}
\open{P}_\mu''(0) = \|\mu\|^2 - \sum_{\alpha \in \mathcal{A}} m_\alpha^2 \quad \text{ and } \quad
\closed{P}_\mu''(0) = \|\mu\|^2 + \sum_{\alpha \in \mathcal{A}} m_\alpha^2.
\end{equation}
\end{lem}
\begin{proof}
We define $ \theta := \sum_{\alpha \in \mathcal{A}} m_\alpha^2 $.
From \cref{prop:total-run-funcs}, we know
\[ \open{P}_\mu(z) = \sum_{n=0}^\infty \open{L}_n z^n = e^{m_d z} \prod_{\alpha \in \mathcal{A}} (1 + m_\alpha z) \]
for all $ z \in \mathbb{C} $ and the product converges uniformly on bounded sets.
So we define $ \open{Q}(z) := \prod_{\alpha \in \mathcal{A}} (1 + m_\alpha z) $
and take the derivative
\[
\open{Q}'(z) = \sum_{\alpha \in \mathcal{A}} m_\alpha \prod_{\substack{\beta \in \mathcal{A} \\ \beta \neq \alpha}} (1 + m_\beta z)
\quad \text{ so the constant term is } \quad
\open{Q}'(0) = \sum_{\alpha \in \mathcal{A}} m_\alpha = \|\mu\| - m_d.
\]
The second derivative and its constant term will be
\[ \begin{split}
\open{Q}''(z) &= \sum_{\substack{\alpha, \beta \in \mathcal{A} \\ \alpha \neq \beta}} m_\alpha m_\beta
\prod_{\substack{\gamma \in \mathcal{A} \\ \gamma \neq \alpha, \beta}} (1 + m_\gamma z) \quad \text{ and} \\
\open{Q}''(0) &= \sum_{\substack{\alpha, \beta \in \mathcal{A} \\ \alpha \neq \beta}} m_\alpha m_\beta
= \sum_{\alpha, \beta \in \mathcal{A}} m_\alpha m_\beta - \sum_{\substack{\alpha, \beta \in \mathcal{A} \\ \alpha = \beta}} m_\alpha m_\beta \\
&= \left(\sum_{\alpha \in \mathcal{A}} m_\alpha\right)^2 - \sum_{\alpha \in \mathcal{A}} m_\alpha^2 = (\|\mu\| - m_d)^2 - \theta.
\end{split} \]
Thus, the second derivative of $ \open{P}_\mu $ at zero is
\[ \begin{split}
\open{P}''_\mu(0) &= m_d^2 e^{m_d \cdot 0} \; \open{Q}(0) + 2 m_d e^{m_d \cdot 0} \; \open{Q}'(0) + e^{m_d \cdot 0} \; \open{Q}''(0) \\
&= m_d^2 + 2 m_d (\|\mu\| - m_d) + (\|\mu\| - m_d)^2 - \theta \\
&= m_d^2 + 2 \|\mu\| m_d - 2 m_d^2 + \|\mu\|^2 - 2 \|\mu\| m_d + m_d^2 - \theta = \|\mu\|^2 - \theta.
\end{split} \]
proving the first part of \cref{eq:coef2}.
\bigskip

Now, we will address the non-strict case.
We define $ \closed{Q}(z) := \prod_{\alpha \in \mathcal{A}} \frac{1}{1 - m_\alpha z} $
for $ z \in \mathbb{C} $ sufficiently close to zero and the first derivative near zero is
\[
\closed{Q}'(z) = \sum_{\alpha \in \mathcal{A}} \frac{m_\alpha}{(1 - m_\alpha z)^2} \prod_{\substack{\beta \in \mathcal{A} \\ \beta \neq \alpha}} \frac{1}{1 - m_\beta z}
\quad \text{ so the constant term is } \quad
\closed{Q}'(0) = \sum_{\alpha \in \mathcal{A}} \frac{m_\alpha}{(1 - 0)^2} = \|\mu\| - m_d.
\]
For the second derivative, we may differentiate the $ \alpha $ term again or differentiate a different term
\[
\closed{Q}''(z) = \sum_{\alpha \in \mathcal{A}} \left[
\frac{2 m_\alpha^2}{(1 - m_\alpha z)^3} \prod_{\substack{\beta \in \mathcal{A} \\ \beta \neq \alpha}} \frac{1}{1 - m_\beta z}
+ \sum_{\substack{\beta \in \mathcal{A} \\ \beta \neq \alpha}} \frac{m_\alpha m_\beta}{(1 - m_\alpha z)^2 (1 - m_\beta z)^2}
\prod_{\substack{\gamma \in \mathcal{A} \\ \gamma \neq \alpha, \beta}} \frac{1}{1 - m_\gamma z}
\right]
\]
and its constant term will be
\[ \begin{split}
\closed{Q}''(0) &= \sum_{\alpha \in \mathcal{A}} \left[ \frac{2 m_\alpha^2}{(1 - 0)^3}
+ \sum_{\substack{\beta \in \mathcal{A} \\ \beta \neq \alpha}} \frac{m_\alpha m_\beta}{(1 - 0)^2 (1 - 0)^2} \right]
= 2 \left(\sum_{\alpha \in \mathcal{A}} m_\alpha^2\right) + \left(\sum_{\substack{\alpha, \beta \in \mathcal{A} \\ \alpha \neq \beta}} m_\alpha m_\beta \right) \\
&= \sum_{\alpha \in \mathcal{A}} m_\alpha^2 + \sum_{\alpha, \beta \in \mathcal{A}} m_\alpha m_\beta = \theta + (\|\mu\| - m_d)^2.
\end{split} \]
Hence, evaluating $ \closed{P}''_\mu(0) $ we get that
\[ \begin{split}
\closed{P}''_\mu(0) &= m_d^2 e^{m_d \cdot 0} \; \closed{Q}(0) + 2 m_d e^{m_d \cdot 0} \; \closed{Q}'(0) + e^{m_d \cdot 0} \; \closed{Q}''(0) \\
&= m_d^2 + 2 m_d (\|\mu\| - m_d) + \theta + (\|\mu\| - m_d)^2 \\
&= m_d^2 + 2 \|\mu\| m_d - 2 m_d^2 + \theta + \|\mu\|^2 - 2 \|\mu\| m_d + m_d^2 = \|\mu\|^2 + \theta
\end{split} \]
proving the second part of \cref{eq:coef2}.
\end{proof}

\begin{proof}[Proof of \cref{thm:total-length-pgf} and \cref{coro:total-length-stats}]
It simply remains to apply \cref{thm:length-pgf} and \cref{coro:length-stats}
to the results from \cref{prop:total-run-funcs} and \cref{lem:total-coef2}.
Here, we are assuming $ \|\mu\| = 1 $.
We get the PGFs by substituting \cref{eq:total-run-funcs} into \cref{eq:init-length-pgf} and \cref{eq:length-pgf}.
The expected values of $ \open{N}_0 $ and $ \closed{N}_0 $
follow immediately from applying \cref{eq:init-length-mean}
to the formulae in \cref{prop:total-run-funcs}.
For $ i \geq 1 $, we get $ \expectWhen{\open{N}_i}{\open{S}_i} $ and $ \expectWhen{\closed{N}_i}{\closed{S}_i} $
from applying \cref{eq:length-mean} to \cref{lem:total-coef2}.
\medskip

Next, we will prove the formulae for the variances.
For $ z $ sufficiently close to one, the derivative of $ P_\mu(z) $ is
\[ \begin{split}
\open{P}_\mu'(z) &= m_d e^{m_d z} \prod_{\alpha \in \mathcal{A}} (1 + m_\alpha z)
+ e^{m_d z} \sum_{\alpha \in \mathcal{A}} m_\alpha
\prod_{\substack{\beta \in \mathcal{A} \\ \beta \neq \alpha}} (1 + m_\beta z)
= \left(e^{m_d z} \prod_{\alpha \in \mathcal{A}} (1 + m_\alpha z)\right)
\left(m_d + \sum_{\alpha \in \mathcal{A}} \frac{m_\alpha}{1 + m_\alpha z} \right)
\quad \text{ and} \\
\closed{P}_\mu'(z) &= m_d e^{m_d z} \prod_{\alpha \in \mathcal{A}} \frac{1}{1 - m_\alpha z}
+ e^{m_d z} \sum_{\alpha \in \mathcal{A}} \frac{m_\alpha}{(1 - m_\alpha z)^2}
\prod_{\substack{\beta \in \mathcal{A} \\ \beta \neq \alpha}} \frac{1}{1 - m_\beta z}
= \left(e^{m_d z} \prod_{\alpha \in \mathcal{A}} \frac{1}{1 - m_\alpha z}\right)
\left(m_d + \sum_{\alpha \in \mathcal{A}} \frac{m_\alpha}{1 - m_\alpha z} \right).
\end{split} \]
Then, the derivatives at one are
\[ \begin{split}
\open{P}_\mu'(1)
&= \open{P}_\mu(1) \left(1 - \sum_{\alpha \in \mathcal{A}} m_\alpha + \sum_{\alpha \in \mathcal{A}} \frac{m_\alpha}{1 + m_\alpha} \right)
= \open{P}_\mu(1) \left(1 + \sum_{\alpha \in \mathcal{A}} \left( \frac{m_\alpha}{1 + m_\alpha} - m_\alpha\right) \right)
= \open{P}_\mu(1) \left(1 + \sum_{\alpha \in \mathcal{A}} \frac{-m_\alpha^2}{1 + m_\alpha} \right)
\quad \text{ and} \\
\closed{P}_\mu'(1)
&= \closed{P}_\mu(1) \left(1 - \sum_{\alpha \in \mathcal{A}} m_\alpha + \sum_{\alpha \in \mathcal{A}} \frac{m_\alpha}{1 - m_\alpha} \right)
= \closed{P}_\mu(1) \left(1 + \sum_{\alpha \in \mathcal{A}} \left( \frac{m_\alpha}{1 - m_\alpha} - m_\alpha\right) \right)
= \closed{P}_\mu(1) \left(1 + \sum_{\alpha \in \mathcal{A}} \frac{m_\alpha^2}{1 - m_\alpha} \right).
\end{split} \]
Then, the variances will be
\[ \begin{split}
\var{\open{N}_0} &= \open{P}_\mu(1) - \open{P}_\mu(1)^2 + 2 \open{P}_\mu'(1)
= \left(e^{m_d} \prod_{\alpha \in \mathcal{A}} (1 + m_\alpha) \right)
\left(3 - e^{m_d} \prod_{\alpha \in \mathcal{A}} (1 + m_\alpha) - 2 \sum_{\alpha \in \mathcal{A}} \frac{m_\alpha^2}{1 + m_\alpha} \right)
\quad \text{ and} \\
\var{\closed{N}_0} &= \closed{P}_\mu(1) - \closed{P}_\mu(1)^2 + 2 \closed{P}_\mu'(1)
= \left(e^{m_d} \prod_{\alpha \in \mathcal{A}} \frac{1}{1 - m_\alpha}\right)
\left(3 - e^{m_d} \prod_{\alpha \in \mathcal{A}} \frac{1}{1 - m_\alpha} + 2 \sum_{\alpha \in \mathcal{A}} \frac{m_\alpha^2}{1 - m_\alpha} \right).
\end{split} \]
For $ i \geq 1 $, we can simply substitute our formulae for $ P_\mu''(0) $ and $ P_\mu(1) $ into \cref{eq:length-var}.
Again by substitution, we get \cref{eq:total-init-length-pgf} and \cref{eq:total-length-pgf}.
\end{proof}

\begin{proof}[Proof of \cref{thm:series-parallel}]
Suppose $ \mu $ is a finite non-negative measure on a countably series-parallel partial order $ T $.
If $ T $ is a total order then apply \cref{prop:total-run-funcs} to get the run functions.
Otherwise, $ T $ is a countable series or parallel composition.
So there exists a collection $ \{T_\gamma\}_{\gamma \in \Gamma} $ with countable $ \Gamma $
where $ T = \bigotimes_\gamma T_\gamma $ or $ T = \bigoplus_\gamma T_\gamma $.
Because $ \Gamma $ is countable, we can use $ \sigma $-additivity to say
for any measurable $ E $,
\[
\sum_{\gamma \in \Gamma} \left( \mu \big|_{T_\gamma} \right)(E) =
\sum_{\gamma \in \Gamma} \mu(E \cap T_\gamma) = \mu(E).
\]
Thus, $ \mu = \bigotimes_\gamma \mu \big|_{T_\gamma} $ when $ T $ is a series composition
and $ \mu = \bigoplus_\gamma \mu \big|_{T_\gamma} $ when $ T $ is a parallel composition.
Therefore, by \cref{prop:comb-funcs},
\[ \begin{split}
P_\mu(z) &= \prod_\gamma P_{\mu |_{T_\gamma}}(z) \; \text{ when $ T $ is a series composition and} \\
P_\mu(z) - 1 &= \sum_\gamma \left( P_{\mu |_{T_\gamma}}(z) - 1 \right) \; \text{ when $ T $ is a parallel composition}
\end{split} \]
for all $ z \in \mathbb{C} $ within the radius of convergence.
\end{proof}

\filbreak
\providecommand{\bysame}{\leavevmode\hbox to3em{\hrulefill}\thinspace}
\providecommand{\MR}{\relax\ifhmode\unskip\space\fi MR }
\providecommand{\MRhref}[2]{%
  \href{http://www.ams.org/mathscinet-getitem?mr=#1}{#2}
}
\providecommand{\href}[2]{#2}


\begin{thebibliography}{10}

\bibitem{AN}
Mohammad Ahsanullah and Valery~B. Nevzorov, \emph{Records via probability
  theory}, first ed., Atlantis Studies in Probability and Statistics, Atlantis
  Press Paris, 2015.

\bibitem{DAnd}
D\'{e}sir\'{e} Andr\'{e}, \emph{\'{E}tude sur les maxima, minima et
  s\'{e}quences des permutations}, Ann. Sci. \'{E}cole Norm. Sup. (3)
  \textbf{1} (1884), 121--134. \MR{1508737}

\bibitem{GAnd}
George~E. Andrews, \emph{{$q$}-series: their development and application in
  analysis, number theory, combinatorics, physics, and computer algebra}, CBMS
  Regional Conference Series in Mathematics, vol.~66, Conference Board of the
  Mathematical Sciences, Washington, DC; by the American Mathematical Society,
  Providence, RI, 1986. \MR{858826}

\bibitem{Bon}
Mikl\'{o}s B\'{o}na, \emph{Combinatorics of permutations}, Discrete Mathematics
  and its Applications (Boca Raton), Chapman \& Hall/CRC, Boca Raton, FL, 2004,
  With a foreword by Richard Stanley. \MR{2078910}

\bibitem{Can}
E.~Rodney Canfield and Herbert~S. Wilf, \emph{Counting permutations by their
  alternating runs}, J. Combin. Theory Ser. A \textbf{115} (2008), no.~2,
  213--225. \MR{2382512}

\bibitem{Chen}
William Y.~C. Chen and Amy~M. Fu, \emph{A grammatical calculus for peaks and
  runs of permutations}, J. Algebraic Combin. \textbf{57} (2023), no.~4,
  1139--1162. \MR{4588127}

\bibitem{Key}
Eric~S. Key, \emph{On the number of records in an iid discrete sequence}, J.
  Theor. Probab. \textbf{18} (2005), 99 -- 107.

\bibitem{Knu2}
Donald~E. Knuth, \emph{The art of computer programming. {V}ol. 2}, third ed.,
  Addison-Wesley, Reading, MA, 1998, Seminumerical algorithms. \MR{3077153}

\bibitem{Knu3}
\bysame, \emph{The art of computer programming. {V}ol. 3}, second ed.,
  Addison-Wesley, Reading, MA, 1998, Sorting and searching. \MR{3077154}

\bibitem{Lyo}
Russell Lyons, \emph{Strong laws of large numbers for weakly correlated random
  variables}, Michigan Math. J. \textbf{35} (1988), no.~3, 353--359.
  \MR{978305}

\bibitem{Ma}
Shi-Mei Ma, \emph{Enumeration of permutations by number of alternating runs},
  Discrete Math. \textbf{313} (2013), no.~18, 1816--1822. \MR{3068991}

\bibitem{Moh}
Rolf~H. M\"{o}hring, \emph{Computationally tractable classes of ordered sets},
  Algorithms and order ({O}ttawa, {ON}, 1987), Kluwer Acad. Publ., Dordrecht,
  1989, pp.~105--193. \MR{1037783}

\bibitem{Nag}
H.~N. Nagaraja, \emph{Order statistics from discrete distributions}, Statistics
  \textbf{23} (1992), no.~3, 189--216.

\bibitem{Ren}
Alfr\'{e}d R\'{e}nyi, \emph{Th\'{e}orie des \'{e}l\'{e}ments saillants d'une
  suite d'observations}, Ann. Fac. Sci. Univ. Clermont-Ferrand \textbf{8}
  (1962), 7--13. \MR{286162}

\bibitem{Sch}
Bernd S.~W. Schr\"{o}der, \emph{Ordered sets}, Birkh\"{a}user Boston, Inc.,
  Boston, MA, 2003, An introduction. \MR{1944415}

\end{thebibliography}
\end{document}